\documentclass[11pt]{amsart}
\usepackage{amssymb,amsmath}
\usepackage{caption}

\usepackage{epsfig}

\def\Bbb{\mathbb}

\def\eb{\varepsilon}

\def\R {\mathbb{R}}

\def\M {{\mathcal M}}

\def\<{\left<}
\def\>{\right>}

\def\Dx{\Delta_x}
\def\({\left(}
\def\){\right)}

\def\Sum{\sum\nolimits'}

\def\erf{\operatorname{erf}}

\newtheorem{proposition}{Proposition}[section]
\newtheorem{theorem}[proposition]{Theorem}
\newtheorem{corollary}[proposition]{Corollary}
\newtheorem{lemma}[proposition]{Lemma}

\theoremstyle{definition}

\newtheorem{remark}[proposition]{Remark}

\numberwithin{equation}{section}

\def \no#1#2#3 {{\bf #1} (#3), #2.}
\def \eds#1#2#3 {#1, #2, #3.}
\begin{document}
\title[] {Asymptotic expansions and extremals for the critical Sobolev and Gagliardo-Nirenberg inequalities on a torus}

\author[] {Michele Bartuccelli, Jonathan Deane and Sergey Zelik}


\address{University of Surrey, Department of Mathematics, \newline
Guildford, GU2 7XH, United Kingdom.}
\email{J.Deane@surrey.ac.uk\\M.Bartuccelli@surrey.ac.uk\\s.zelik@surrey.ac.uk}

\begin{abstract} We give a comprehensive study of  interpolation inequalities for periodic functions with zero mean, including the existence of and the asymptotic expansions for the extremals, best constants, various remainder terms, etc. Most attention is paid to the critical (logarithmic) Sobolev inequality in the two-dimensional case, although a number of results concerning the best constants in the
algebraic case and different space dimensions are also obtained.
\end{abstract}
\maketitle
\tableofcontents

\section{Introduction}\label{s0}
We study  the following critical Sobolev inequality (see \cite{BG,Gr}):
\begin{equation}\label{0.main}
\|u\|_{L^\infty(\Omega)}^2\le \|u\|_{H^1(\Omega)}^2\(C_1\log\frac{\|u\|_{H^2(\Omega)}^2}{\|u\|_{H^1(\Omega)}^2}+C_2\)
\end{equation}
in the particular case when $\Omega$ is a two dimensional torus. This inequality, which can be formally considered
as a limit case ($l\to1$, $n=2$, d=2) of the algebraic inequality of the Gagliardo-Nirenberg type:
\begin{multline}\label{0.alg}
\|u\|_{L^\infty(\Omega)}\le C_\Omega(l,n)\|(-\Delta)^{-l/2}u\|_{L^2(\Omega)}^\theta\|(-\Delta)^{-n/2}u\|_{L^2(\Omega)}^{1-\theta},\\ \theta=\frac{(n-d/2)}{n-l}, \ \ n>d/2>l, \ \ \Omega\subset\subset\R^d
\end{multline}
is known to be very useful in many problems related to partial differential equations and mathematical physics. For instance, it is used for obtaining  best known upper bounds for the attractor dimension of the Navier-Stokes system on a 2D torus (see e.g. \cite{Tem}), for proving the uniqueness of weak solutions for von Karman-type equations arising in elasticity (see \cite{Chu} and references therein) as well as for the so-called hyperbolic relaxation of the 2D Cahn-Hilliard equation (see \cite{GSZ}) or the 2D Klein-Gordon equation with exponential nonlinearity (see \cite{Ibr1}). We
mention also that a slightly different logarithmic inequality is used at a
crucial point in the  proof of the global existence of strong solutions of the 2D Euler equations (see \cite{Yu}).
\par
Note that nowadays most classical inequalities of   Gagliardo-Nirenberg  type can be relatively easily verified using interpolation theory (see e.g., \cite{Tri}). However, the best constants in those inequalities as well as the existence and the analytic structure of the extremals is a
much more delicate and interesting question, which is far from being completely understood despite
persistent interest in the problem and the many interesting results obtained during the last 50
years; see \cite{Au,Bir,BL,C,FF,Ibr,Il1,Il2,M85,Tal94,Ros,Wad}
and references therein. The most studied is, of course, the case of the whole space
$\Omega=\R^d$,
 more or less complete results are available in two cases: where the inequality does not contain derivatives of order higher than one or in the Hilbert case. In the first case the rearrangement technique works and reduces the problem to the one-dimensional case and in the second case one can use the Parsevsal equality. In particular, as proved in \cite{Il2}, the best constant in \eqref{0.alg} for the case $\Omega=\R^d$ is
\begin{equation}\label{0.const}
c_{\R^d}(l,n)=\(\frac{\pi\omega(d)}{(2\pi)^d\sin\frac{d-2l}{2(n-l)}}\cdot\(\frac1{(d-2l)^{d-2l}(2n-d)^{2n-d}}\)^{\frac1{2(n-l)}}\)^{1/2},
\end{equation}
where $\omega(d)=\frac{2\pi^{d/2}}{\Gamma(d/2)}$ is the surface area of the $(d-1)$-dimensional
sphere. In addition, the extremal  function $u_*\in (-\Delta)^{-n/2}L^2(\R^d)\cap (-\Delta)^{-l/2}L^2(\R^d)$  exists and is unique up to a shift and scaling $u_*(x)\to \alpha u_*(\beta x-x_0)$, $\alpha,\beta\in\R$, $x_0\in\R^d$; $u_*$ is given by
\begin{equation}\label{0.extr}
u_*(x)=\frac1{(2\pi)^{d/2}}\int\frac1{|\xi|^{2n}+|\xi|^{2l}}e^{ix\xi}\,dx.
\end{equation}
The situation becomes more complicated in the case where $\Omega$ is a bounded domain of $\R^d$,
even for the algebraic inequality \eqref{0.alg} with Hilbert norms on the right-hand side. To the best of our knowledge, two  different scenarios are possible here.
In the first case, the sharp constant $c_\Omega(l,n)$ coincides with $c_{\R^d}(l,n)$, but in contrast to the case of $\R^d$, there are no exact extremals and the approximative extremals can be constructed by the proper scaling and cutting of  the function \eqref{0.extr}. This case is realized, for instance, if Dirichlet boundary conditions are posed; or if $\Omega=\Bbb S^1$ is a circle (periodic boundary conditions) and $n=0$; or if $\Omega=\Bbb S^d$ is a higher dimensional sphere ($d=2$ and $\Delta$ is a Laplace-Beltrami operator) with $n=0$ and $l\le 7$. See \cite{Il1,Il2} for details. In the present paper, we show that it is also true for the tori $\Omega=\Bbb T^2$ and $\Omega=\Bbb T^3$ if $l=0$ and $n$ is not too large
--- see Section \ref{s4}. In addition, in that case, inequality \eqref{0.alg} can be improved by adding an extra lower order term in the spirit of Brezis and Lieb (see \cite{BL}). In particular, as shown in Section \ref{s4}, the following inequality
\begin{equation}\label{0.l2h2}
\|u\|_{L^\infty(\Bbb T^2)}^2\le \frac14\|u\|_{L^2(\Bbb T^2)}\cdot\|\Dx u\|_{L^2(\Bbb T^2)}-\frac1{2\pi^2}\|u\|^2_{L^2(\Bbb T^2)}
\end{equation}
holds for all $2\pi\times2\pi$-periodic functions with zero mean. However, even for this improved inequality the exact extremal functions do not exist, and
further improvements can be obtained.
\par
In the second case, the sharp constant in \eqref{0.alg} is {\it strictly} larger than the analogous constant in $\R^d$:
\begin{equation}\label{0.bad}
c_{\Omega}(l,n)>c_{\R^d}(l,n)
\end{equation}
and there is/are exact extremal function(s) for \eqref{0.alg} in $H^n(\Omega)$. In
particular, this holds for $\Omega=\Bbb S^2$ with $l=0$, $n\ge8$, see \cite{Il2} (see also \cite{Il1} for the analogous effect for the slightly different inequality in the one-dimensional case). In that case, the constant $c_\Omega(l,n)$ can be found only numerically as a root of a transcendental equation.
\par
It was conjectured by Ilyin that the analogous effect holds on
multi-dimensional tori $\Omega=\Bbb T^d$, $d>1$ and $l=0$. In the present paper,
we verify that this conjecture is indeed true and that the inequality \eqref{0.bad}
holds for $\Omega=\Bbb T^2$ (with zero mean) for $l=0$ and $n=10$. In addition, we establish the following  3D analogue of \eqref{0.l2h2}:
\begin{equation}\label{0.l2h23d}
\|u\|^2_{L^\infty(\Bbb T^3)}\le \frac{\sqrt{2\sqrt3}}{6\pi}\|u\|_{L^2(\Bbb T^3)}^{1/2}\|\Delta u\|_{L^2(\Bbb T^3)}^{3/2}-K\|u\|^2_{L^2(\Bbb T^3)},
\end{equation}
where the sharp constant $\frac{\sqrt{2\sqrt3}}{6\pi}$ still {\it coincides} with the analogous constant in the whole space $\R^3$, but nevertheless \eqref{0.l2h23d} possesses an exact extremal function and the best value for the second constant $K$ can be found only numerically
($K\sim \frac{0.996}{2\pi^3}$, see Section \ref{s4}).
\par
Let us now return to the limit {\it logarithmic} inequality \eqref{0.main}. This case looks more difficult than the algebraic one, in particular, since it is a priori not clear  whether or not the transcendental function $\delta\to C_1\log\delta+C_2$
($\delta:=\frac{\|u\|_{H^2(\Omega)}^2}{\|u\|_{H^1(\Omega)}^2}$) on the right-hand side of \eqref{0.main} is optimal. Indeed, a detailed study of the slightly different logarithmic inequality
\begin{equation}\label{0.holder}
\|u\|_{L^\infty(\Omega)}^2\le \|u\|_{H^1(\Omega)}^2\(C_1'\log\frac{\|u\|_{C^\alpha(\Omega)}^2}{\|u\|_{H^1(\Omega)}^2}+C_2'\),
\end{equation}
where the $H^2$-norm is replaced by the H\"older norm with $\alpha\in(0,1)$ is given in recent papers \cite{Bir,Ibr} for the case where $\Omega$ is a unit ball and the function $u$ satisfies the {\it Dirichlet} boundary conditions (see also \cite{Wad,MSW,MSW1}). As shown there, $C_1>\frac1{4\pi\alpha}$ and, in order to be able to take $C_1'=\frac1{4\pi\alpha}$, an extra double logarithmic corrector
 ($\delta\to\log\log\delta$) is required. Thus, based on that result and on the interpolation inequality
 $$
 \|u\|_{C^\alpha}^2\le C\|u\|_{H^1}^{2(1-\alpha)}\|u\|_{H^2}^{2\alpha},
 $$
one may expect the following improved version of \eqref{0.main}
\begin{equation}\label{0.dlog}
\|u\|_{L^\infty(\Bbb T^2)}^2\le \frac1{4\pi}\|\nabla u\|_{L^2(\Bbb T^2)}^2\(\log\frac{\|\Delta u\|_{L^2(\Bbb T^2)}^2}{\|\nabla u\|_{L^2(\Bbb T^2)}^2}+
\log\(1+\log\frac{\|\Delta u\|_{L^2(\Bbb T^2)}^2}{\|\nabla u\|_{L^2(\Omega)}^2}\)+L\),\ L>0
\end{equation}
to be optimal for the case of $2\pi\times2\pi$-periodic functions $u$ with zero mean. Note that the
analysis presented in~\cite{Bir,Ibr} is based on the reducing the problem to the radially symmetric case via the rearrangement technique, and use of the Dirichlet boundary conditions
is essential, so it is not clear how to extend it --- either to the case of the torus or to the case of the $H^2$-norm. Nevertheless, as we will see below, inequality \eqref{0.dlog} is true for the
properly chosen constant $L$ (which can be found numerically as a solution of a
transcendental equation: $L\sim 2.15627$). In addition, there exist exact extremal functions for this inequality;
see Section \ref{s2}.
\par
The main aim of the present paper is to introduce a general scheme which allows the
analysis of inequalities \eqref{0.main},\eqref{0.alg} and \eqref{0.dlog} at least on tori, from
a unified point of view, and to illustrate it in the most complicated logarithmic case (although
nontrivial applications to the algebraic case will be also considered). One of the important features of our approach is that, in contrast to, say, \cite{Ibr,Il2} (and similarly to \cite{Bir}), the concrete form of the right-hand sides in those inequalities is not a priori postulated, but appears a posteriori as a result of computations. Indeed, instead of \eqref{0.dlog}, we consider the following variational problem with constraints:
\begin{equation}\label{0.var}
\frac{\|u\|^2_{L^\infty}}{\|\nabla u\|^2_{L^2}}\to\max,\ \ u\in H^2(\Bbb T^2),\ \int_{\Bbb T^2}u(x)\,dx=0,\ \ \frac{\|\Delta u\|_{L^2}^2}{\|\nabla u\|_{L^2}^2}=\delta
\end{equation}
and prove that, for every $\delta>0$, this problem has a unique (up to shifts, scaling and
alternation of sign) solution $u_\mu(x)$,
\begin{equation}\label{0.extmu}
u_\mu(x)=\sum_{k\in\Bbb Z^2-\{0\}}\frac{e^{ik\cdot x}}{k^2(1+\mu k^2)},\ \ k^2:=k_1^2+k_2^2
\end{equation}
(compare with \eqref{0.extr}) and the parameter $\mu$ can be found, in a unique way, as a solution of the equation
\begin{equation}\label{0.constr}
\frac{\|\Delta u_\mu\|_{L^2}^2}{\|\nabla u_\mu\|_{L^2}^2}=\delta.
\end{equation}
Let us denote the maximum in \eqref{0.var} by $\Theta(\delta)$; as we will see,
$\Theta$ is a real analytic function of $\delta$. Then, the following inequality holds:
\begin{equation}\label{0.th}
\|u\|_{L^\infty(\Bbb T^2)}^2\le\|\nabla u\|^2_{L^2(\Bbb T^2)}\Theta\(\frac{\|\Delta u_\mu\|_{L^2}^2}{\|\nabla u_\mu\|_{L^2}^2}\)
\end{equation}
and by definition $\Theta$ is the least possible function in this inequality. Thus, inequality \eqref{0.th} can be considered as an optimal version of \eqref{0.main} and \eqref{0.dlog}. However, inequality \eqref{0.th} is not convenient for applications since the function $\Theta$ is given in a very implicit form through the lattice sums \eqref{0.extmu} which, to the best of our knowledge, cannot be expressed in closed form through the elementary functions (in contrast to the case of inequality \eqref{0.holder} in a unit ball, see \cite{Bir}) and, in addition, direct numerical computation of them is not easy especially for large $\delta$ (small $\mu$) due to very slow rate of convergence.
\par
In order to overcome this problem, we have found the asymptotic expansions for the function $\Theta(\delta)$ as $\delta\to\infty$. Namely, we have proved that the function $\Theta(\delta)$ coincides up to exponentially small terms (of order $O(e^{-2\pi\delta^{1/2}}))$ with the function $\Theta_0(\delta)$ given by the following parametric expression:
\begin{equation}\label{0.th0}
\Theta_0=\frac1{4\pi^2}\cdot \frac{(\pi\log\frac1\mu+\beta+\mu)^2}{\pi\log\frac1\mu+\beta-\pi+2\mu},\ \ \delta=\frac{\frac\pi\mu-1}{\pi\log\frac1\mu+\beta-\pi+2\mu},
\end{equation}
where $\beta:=\pi\(2\gamma+2\log2+3\log\pi-4\log\Gamma(1/4)\)$, $\gamma$ is the Euler constant
and $\Gamma(z)$ is the Euler gamma function. In particular,
$$
\Theta_0(\delta)=\frac1{4\pi}\log\delta+\frac1{4\pi}\log\log\delta+\frac{\beta+\pi}{4\pi^2}+O_{\delta\to\infty}(1),
$$
which justifies inequality \eqref{0.dlog}  and shows that the constant $L\ge L_\infty:= \frac{\beta+\pi}{\pi}$.  In practice, the numerics shows that $L\ge L_{opt}>L_\infty$ --- see Section \ref{s2}.
\par
In addition, combining the analytic asymptotic expansions for $\Theta(\delta)$ with
numerical simulation for relatively small $\delta$, we show that
\begin{equation}\label{0.good}
\Theta(\delta)\le\Theta_0(\delta)
\end{equation}
{\it for all} $\delta\ge1$. Thus, the much simpler function $\Theta_0$ can be used
instead of $\Theta$ in the right-hand side of \eqref{0.th}. Actually, $\Theta_0$ gives a
reasonable approximation to $\Theta$ for all values of $\delta$. For
instance, for $\delta = 1, 2$ and 4, respectively, we have $\Theta$ (resp. $\Theta_0$) $=
0.10134 (0.17797), 0.26651 (0.26660)$ and $0.35112 (0.35112)$.
\par
The paper is organised as follows. The proof of the existence of the conditional extremals
for problem \eqref{0.var} as well as analytical formulae for them in terms of the lattice
sums, are given in Section \ref{s1}.
\par
The key asymptotic expansions for the lattice sums involving the parametric expression for $\Theta(\delta)$, as well as for the extremals $u_\mu(x)$, are presented in Section \ref{s2}. Based on these expansions, we check the validity of inequality \eqref{0.dlog} as well as the estimate
\eqref{0.good}.
\par
The elementary approaches to the logarithmic inequality \eqref{0.main} are analyzed in Section \ref{s3}.
Actually, there are at least two known ways to prove this inequality without studying  the corresponding extremal problem: one of them is based on the embedding $H^{1+\eb}\subset C$ with further optimization
of the exponent $\eb>0$ (see e.g. \cite{rubbish}), and the other more classical one (which has been factually used in the original paper \cite{BG})  splits the function $u$ into lower and higher
Fourier modes and estimates them via the $H^1$ and $H^2$-norms respectively. Based on the above asymptotic analysis, we show that the second method is preferable and allows to find the correct expressions for the two leading terms in the asymptotic expansions of the function $\Theta$.
\par
The application of our approach to the  simpler algebraic case  \eqref{0.alg} with $l=0$ and arbitrary  space dimension $d$ is
considered in Section \ref{s4}. We establish here the following improved version of \eqref{0.alg}
\begin{equation}\label{0.alg1}
\|u\|^2_{C(\Bbb T^d)}\le c_d(n)\|u\|_{L^2}^{2-d/n}\|(-\Dx)^{n/2}u\|_{L^2}^{d/n}-K_d(n)\|u\|^2_{L^2},
\end{equation}
where $c_d(n)=c_{\R^d}(0,n)$ for all $n\in\Bbb N$ such that $2n-d>0$ and the constant $K_d(n)$ may
be either positive or negative.
We prove that in the one dimensional case this constant is strictly positive, but it may be
either positive or negative in the multi-dimensional case, in depending on $n$.
We present also combined analytical/numerical results for the constants $K_d(n)$ for $d$ and $n$
not large. In particular, inequalities \eqref{0.l2h2}, \eqref{0.l2h23d} mentioned above, as
well as the 1D inequalities
 $$
 \|u\|^2_{C(\Bbb T^1)}\le \|u\|_{L^2}\|u'\|_{L^2}-\frac1\pi\|u\|^2_{L^2},\ \
  \|u\|^2_{C(\Bbb T^1)}\le \frac{\sqrt2}{\sqrt[4]{27}}\|u\|_{L^2}^{3/2}\|u''\|_{L^2}^{1/2}-\frac2{3\pi}\|u\|^2_{L^2},
 $$
are verified there.
\par
The large $n$ limit of the inequality \eqref{0.alg1} is studied in Section \ref{s5}. The results of this section clarify the nature of oscillations in the analog of the function $\delta\to\Theta(\delta)$ for that inequality and show the principal difference between the 1D case where the regular oscillations occur (after the proper scaling) and the multi-dimensional case where the oscillations are irregular due to some number theoretic reasons.
\par
Finally, the computation of the integration constant $\beta$ is given in the Appendix.

{\bf Acknowledgement.} The authors would like to thank A. Ilyin for many stimulating discussions and
comments.

\section{Conditional extremals: existence, uniqueness and  analytical expressions}\label{s1}
This section is devoted to the  study of the maximisation problem \eqref{0.var} which we rewrite in the following equivalent form:
\begin{equation}\label{1.var}
\|u\|^2_{C(\Bbb T^2)}\to\sup,\ \ u\in H^2(\Bbb T^2),\ \int_{\Bbb T^2}u(x)\,dx=0,\ \ \|\Delta u\|_{L^2}^2=\delta, \ \
\|\nabla u\|_{L^2}^2=1.
\end{equation}
In addition, we note that problem \eqref{1.var} is invariant with respect to translations $u(x)\to u(x+h)$ and alternation $u(x)\to-u(x)$. Thus, without loss of generality, we may assume that $\|u\|_{C(\Bbb T^2)}=u(0)>0$ and so reduce problem \eqref{1.var} to the following one:
\begin{equation}\label{1.varsimple}
u(0)\to\sup,\ \ u\in H^2(\Bbb T^2),\ \int_{\Bbb T^2}u(x)\,dx=0,\ \ \|\Delta u\|_{L^2}^2=\delta, \ \
\|\nabla u\|_{L^2}^2=1.
\end{equation}
Thus, the function $\Theta$ in \eqref{0.th} can be defined as follows:
\begin{equation}\label{1.theta}
\Theta(\delta):=\sup\bigg\{u(0)^2,\ \ u\in H^2(\Bbb T^2),\ \int_{\Bbb T^2}u(x)\,dx=0,\ \ \|\Delta u\|_{L^2}^2=\delta, \ \
\|\nabla u\|_{L^2}^2=1\bigg\}.
\end{equation}
It is however more convenient to rewrite problem \eqref{1.varsimple} and \eqref{1.theta} in
Fourier space by expanding
\begin{equation}\label{1.f}
u(x)=\frac1{2\pi}\Sum u_ke^{ix\cdot k},
\end{equation}
where $\Sum$ means the sum over the lattice $k\in\Bbb Z^2$ except $k=0$. Using the Parseval
equality, we transform \eqref{1.varsimple} to
\begin{equation}\label{1.varf}
\frac1{2\pi}\Sum u_k\to\sup,\ \ ,\ \ \Sum (k^2)^2 |u_k|^2=\delta, \ \Sum k^2 |u_k|^2=1.
\end{equation}
Finally, we observe that, without loss of generality, we may assume that all $u_k$ in \eqref{1.varf} are {\it real} and {\it nonnegative}.
\par
\begin{lemma}\label{Lem1.exist} For every $\delta\ge1$ there exists an extremal function (maximiser) for problem \eqref{1.varsimple}
(or equivalently, for problem \eqref{1.varf}).
\end{lemma}
\begin{proof} Let $u_n(x)$, $u_n(0)>0$ be a maximising sequence for problem \eqref{1.varsimple}
such that
$$
\Theta(\delta)=\lim_{n\to\infty}u_n(0)^2.
$$
Such a sequence exists if and only if $\delta\ge1$, since under that condition the set
of functions $u\in H^2(\Bbb T^2)$ for which the constraints of \eqref{1.varsimple} are
satisfied is not empty.
Clearly, $u_n$ is bounded in $H^2$ and consequently, without loss of generality, we may assume that $u_n\to u^*$ weakly in $H^2$ (and strongly in $C(\Bbb T^2)$ and in $H^1$). We claim that $u^*$ is
the desired maximiser. Clearly,
 \begin{equation}\label{1.m}
 \Theta(\delta)=u^*(0)^2,\ \ \|\nabla u^*\|^2_{L^2}=1,\ \ \|\Delta u_*\|_{L^2}^2\le\delta.
 \end{equation}
 Thus, we only need to check that the last inequality is in a fact {\it equality}. Assume that it is not true and
 $\|\Dx u_*\|_{L^2}^2=\delta_0<\delta$. Let us fix $k_0\in\Bbb Z^2$ such that $u_{k_0}>0$ take any small $\eb>0$ and $N>|k_0|$ and consider the perturbed function
 $$
 u_{\eb,N}(x)=u_*(x)-\beta e^{ik_0\cdot x}+\eb\sum_{|k_0|<k<N}\frac{e^{ikx}}{|k|^2\log(|k|+1)},
 $$
 where $\beta=\beta(\eb,N)>0$ is chosen in such way that $\|\nabla u_{\eb,N}\|_{L^2}=1$. Using the fact that $\Sum\frac1{|k|^2\log^2(|k|+1)}<\infty$ and that $u_{k_0}>0$, one can easily show that there are positive constants $\eb_0$ and $l$ independent of $N$ such that
 \begin{equation}\label{1.in1}
 \beta(\eb,N)\le l\eb,\ \ \forall \eb\le\eb_0
 \end{equation}
 and all $N$. On the other hand, using \eqref{1.in1} and the fact that $\Sum\frac1{|k|^2\log(|k|+1)}=\infty$, we see that there exists $N_0$ independent of $\eb$ such that
 \begin{equation}\label{1.in3}
 u_{\eb,N}(0)>u_*(0),\ \ \forall N\ge N_0,\ \eb\le\eb_0.
 \end{equation}
 Finally, since
 $$
 \lim_{\eb\to0}\|\Dx u_{\eb,N}\|_{L^2}^2=\delta_0<\delta, \ \ \lim_{N\to\infty}\|\Dx u_{\eb,N}\|_{L^2}=\infty,
 $$
 we may find $N_*>N_0$ and $\eb_*<\eb_0$ such that $\|\Dx u_{\eb_*,N_*}\|_{L^2}^2=\delta$.
This, together with \eqref{1.in3}, shows that
 $$
 \Theta(\delta)>u_*(0)^2
 $$
 which contradicts our choice of function $u_*(x)$ (see \eqref{1.m}) and finishes the proof of the lemma.
\end{proof}
\begin{remark} In particular, the above arguments show that the function $\delta\to\Theta(\delta)$ is strictly increasing.
\end{remark}
We are now ready to state the main result of this section, which gives the existence and uniqueness for the extreme functions of \eqref{1.var}. These functions will be further referred as {\it conditional extremals} for the logarithmic Sobolev inequality considered.

\begin{theorem}\label{Th1.ext} For every fixed $\delta>0$, variational problem \eqref{0.var} has a unique (up to translations, scalings  and alternation) solution
\begin{equation}\label{1.extremal}
u_\mu(x):=\Sum\frac{e^{ik\cdot x}}{k^2(1+\mu k^2)}
\end{equation}
where $\mu=\mu(\delta)\in(-\infty,-1]\cup(0,\infty]$ is determined as the unique solution of
the equation
\begin{equation}\label{1.F}
F(\mu):=\frac{\Sum\frac1{(1+\mu k^2)^2}}{\Sum\frac1{k^2(1+\mu k^2)^2}}=\delta.
\end{equation}
Thus, the desired function $\Theta(\delta)$ possesses the following parametric representation:
\begin{equation}\label{1.t}
\Theta(\mu):=\frac1{4\pi^2}\cdot\frac{\(\Sum\frac1{|k|^2(1+\mu|k|^2)}\)^2}{\Sum\frac1{|k|^2(1+\mu|k|^2)^2}},\ \ \delta(\mu):=\frac{\Sum\frac1{(1+\mu k^2)^2}}{\Sum\frac1{k^2(1+\mu k^2)^2}}
\end{equation}
and $\mu\in(-\infty,-1]\cup(0,\infty]$.
\end{theorem}
\begin{proof} Instead of \eqref{0.var}, we will consider the equivalent problem \eqref{1.varf}. The extremals of that problem can be easily found using Lagrange multipliers. Introducing the
Lagrange function
$$
\mathcal L(u):=\frac1{2\pi}\Sum u_k+A_1\Sum |k|^2u_k^2+A_2\Sum|k|^4u_k^2,\ \ A_1,A_2\in\R,
$$
differentiating it with respect to $u_k$ and using the necessary condition $\frac d{du}\mathcal L(u)=0$ for extremals, we find the following extremals:
\begin{equation}\label{1.elag}
u_k^*=u_{k,A_1,A_2}^*=\frac1{4\pi|k|^2(A_1+A_2|k|^2)},
\end{equation}
where, as usual, the multipliers $A_1$ and $A_2$ should be chosen to satisfy the constraints.
Since we already know (from Lemma \ref{Lem1.exist}) that the maximiser $u_\delta(x)$ exists, its Fourier coefficients should satisfy \eqref{1.elag} for some $A_1$ and $A_2$. Moreover, taking into the
account the fact that the initial variational problem is scaling invariant, we may get rid of one of the multipliers $A_1$ and $A_2$ by introducing $\mu=A_2/A_1$. We will then end up with the
one-parameter family of extremals \eqref{1.extremal} depending on $\mu$ (the case $A_1=0$ is not lost and will correspond below to $\mu=\infty$). Of course, the parameter $\mu$ should be chosen to satisfy the constraints, namely, $\|\Delta u_\mu\|^2_{L^2}/\|\nabla u_\mu\|_{L^2}^2=\delta$.
This gives equation \eqref{1.F}, and representation \eqref{1.t} follows immediately from the definition of $\Theta(\delta)$.
\par
Thus, we only need to verify that the solution of $F(\mu)=\delta$ is unique. To this end, we
first recall that all the Fourier coefficients of the conditional maximiser(s) should be either non-negative or non-positive. This, together with the formula \eqref{1.extremal}, gives
the condition that $\mu\in(-\infty,-1]$ which corresponds to all negative coefficients
and $\mu\in(0,\infty]$ which corresponds to all positive ones.
 Thus, only the values $\mu\in(-\infty,-1]\cup(0,\infty)$ may correspond to the true maximisers and we need not consider the case $\mu\in(-1,0)$. The following Lemma gives the uniqueness of a solution of \eqref{1.F} in this domain of $\mu$.

\begin{lemma}\label{Lem1.mon} The function $\tilde F(\eb):=F(\eb^{-1})$ is continuous (in
fact real analytic),
strictly increasing on $[-1,\infty)$ and
satisfies
\begin{equation}\label{1.FF}
\tilde F(-1)=1,\ \ \lim_{\eb\to+\infty}\tilde F(\eb)=+\infty
\end{equation}
Therefore,  the solution of $F(\mu)=\delta$, $\mu\in(-\infty,-1]\cup(0,\infty)$, exists and is unique for all $\delta\ge1$.
\end{lemma}
\begin{proof}The function $\tilde F(\eb)$ is
$$
\tilde F(\eb)=\frac{\Sum\frac1{(\eb+|k|^2)^2}}{\Sum \frac1{|k|^2(1+\eb|k|^2)^2}}
$$
and, differentiating this with respect to $\eb$, we have
$$
\tilde F'(\eb)=-2\frac{\Sum_k\Sum_l\frac1{l^2(\eb+k^2)^2(\eb+l^2)^2}\(\frac1{\eb+k^2}-\frac1{\eb+l^2}\)}{\(\Sum \frac1{|k|^2(1+\eb|k|^2)^2}\)^2}=2\frac{\Sum_k\Sum_l\frac1{(\eb+k^2)^3(\eb+l^2)^3}\frac{k^2-l^2}{l^2}}{\(\Sum \frac1{|k|^2(1+\eb|k|^2)^2}\)^2}.
$$
The double-double sum in the numerator can be rearranged to contain only positive terms. Indeed, putting together the terms corresponding to the indices $(l,k)$ and $(k,l)$, we see that
$$
\frac1{(\eb+k^2)^3(\eb+l^2)^3}\(\frac{k^2-l^2}{l^2}+\frac{l^2-k^2}{k^2}\)=\frac{(k^2-l^2)^2}{k^2l^2(\eb+k^2)^3(\eb+l^2)^3}>0.
$$
Thus, we only need to verify \eqref{1.FF}. The first assertion is obvious since both nominator and denominator in the definition of $\tilde F$ have simple poles at $\eb=-1$ with the same residue. The second limit is a bit more difficult, but we do not want to prove it here since the detailed analysis of the asymptotic behavior of $F$ as $\mu\to0$ will be given in the next section. Lemma \ref{Lem1.mon} is proved.
\end{proof}
Thus, due to the proven uniqueness of a solution of \eqref{1.F}, the conditional maximizer $u_\mu(x)=u_{\mu(\delta)}(x)$ for the variational problem \eqref{0.var} is also unique and Theorem \ref{Th1.ext} is proved.
\end{proof}
\begin{remark}\label{Rem1.PDE} It is not difficult to see that the extremals $u_\mu(x)$ defined by \eqref{1.extremal} satisfy the following boundary value problem:
\begin{equation}\label{1.PDE}
\Dx (1-\mu\Dx )u_\mu=-4\pi^2\delta(x)+1
\end{equation}
endowed by the periodic boundary conditions (here $\delta(x)$ is a standard Dirac delta-function) and, therefore, are closely related to fundamental solutions for this family of 4th order elliptic differential operators. It can also be derived by applying the method
of Lagrange multipliers directly to problem \eqref{1.varsimple} (without passing to
Fourier space). We will use this fact below in order to find good asymptotic expansions for $u_\mu(x)$.
\end{remark}

\section{Asymptotic expansions}\label{s2}
In this section, we deduce the asymptotic expansions up to exponential order for the lattice sums used in Theorem \ref{Th1.ext}, which are crucial for our approach. Namely, we will analyse the asymptotic behaviour of the following three sums:
\begin{equation}\label{2.sums}
f(\mu)=\Sum\frac1{k^2(1+\mu k^2)},\ \ g(\mu)=\Sum\frac1{k^2(1+\mu k^2)^2},\ \ h(\mu)=\Sum\frac1{(1+\mu k^2)^2}
\end{equation}
as $\mu\to0$. Actually, these sums are closely related to each other:
\begin{equation}\label{2.dif}
f'(\mu)=-h(\mu),\ \ g(\mu)=f(\mu)-\mu h(\mu)
\end{equation}
and therefore, up to a nontrivial integration constant, we only need to study the simplest
function, $h(\mu)$.
\begin{lemma}\label{Lem2.Poi} The function $h(\mu)$ possesses the following asymptotic expansion:
\begin{equation}\label{2.hexp}
h(\mu)=\frac\pi\mu-1+4\pi^2\mu^{-5/4}e^{-2\pi/\sqrt{\mu}}\(1+o_\mu(1)\)
\end{equation}
as $\mu\to\infty$.
\end{lemma}
\begin{proof} The derivation of the expansion is based on the Poisson summation formula. Using
the fact that
\begin{equation}\label{2.poisson}
\sum_{k\in\Bbb Z^2}\varphi_\mu(k)=\sum_{k\in\Bbb Z^2}\widehat\varphi_\mu(2\pi k),\ \ \varphi_\mu(z)=\frac1{(1+\mu z^2)^2},
\end{equation}
$\varphi_\mu(z)=\varphi_1(\mu^{1/2}z)$ and $\widehat{\varphi_\mu}(\xi)=\mu^{-1}\widehat{\varphi_1}(\mu^{-1/2}\xi)$ together with the fact that $\varphi_1$ is analytic in a strip
$|\mbox{Im} z|<1$, we conclude that
$\widehat{\varphi_1}(\xi)$ is {\it exponentially} decaying:
$$
|\varphi_1(\xi)|\le C_\eb e^{-(1-\eb)|x|}, \forall\eb>0,\ \ \xi\in\R^2.
$$
Thus, if we need the asymptotic expansion of the left-hand side of \eqref{2.poisson} up to
exponential order, only the term with $k=0$ is needed in the right-hand side and, therefore, replacing the sum by the corresponding integral gives an exponentially sharp approximation to the lattice sum:
\begin{multline}\label{2.h0}
h(\mu)=\Sum\frac1{(1+\mu k^2)^2}=\sum_{k\in\Bbb Z^2}\varphi_\mu(k)-1=\\=\int_{\R^2}\frac{dx}{(1+\mu |x|^2)^2}-1+O(e^{(-2+\eb)\pi\mu^{-1/2})})=\frac\pi\mu-1+O(e^{(-2+\eb)\pi\mu^{-1/2})})
\end{multline}
for arbitrary $\eb>0$. However, if we need the leading exponentially small term in expansions
like \eqref{2.hexp}, we need
to look at four more terms on the right-hand side of \eqref{2.poisson}, namely, the terms
corresponding to $k=(0,1),(0,-1),(1,0),(-1,0)$ (other terms will decay faster than $e^{(-2+\eb)\sqrt2\pi/\mu^{1/2}}$). To this end, we need to compute the 2D Fourier transform
of the radially symmetric function $\varphi_\mu(|x|)$. This can be done, for instance, by
noting that that Fourier transform is a radially symmetric fundamental solution of the
squared Helmholtz operator
$$
(1-\mu\Dx)^2u=\delta(x).
$$
The radially symmetric solution of this equation can be explicitly written in terms of Bessel functions:
\begin{equation}\label{2.big}
R(|\xi|):=\widehat{\varphi_1}(\xi)=\pi|\xi|\mu^{-3/2}K_1(|\xi|/\sqrt\mu)
\end{equation}
where $K_1$ is standard Bessel  $K$-function of order 1 (see, e.g., \cite{Wat}). Thus, we only need to find the leading term in $R(2\pi\mu^{-1/2})$ as $\mu\to\infty$, which can easily be done by using
known expansions for the Bessel functions (see \cite{Wat}):
$$
R(2\pi\mu^{-1/2})=\pi^2\mu^{-5/4}e^{-2\pi/\sqrt{\mu}}\(1+o_\mu(1)\).
$$
Taking into account the fact that there are four identical terms on the right-hand side
of~\eqref{2.poisson} which correspond to $|k|=1$, we arrive at \eqref{2.hexp} and finish the proof of the lemma.
\end{proof}
As a next step, we need to derive analogous expansions for $f(\mu)$ and $g(\mu)$. However, the trick with the Poisson summation formula is not directly applicable here since the corresponding function $\varphi$ will have singularity at $x=0$ and, as we will see below, this leads to an extra
residual-type term in the expansions. Instead, we will use relations \eqref{2.dif} in order to
find the expansions for $f$ and $g$ up to an integration constant.

\begin{corollary}\label{Cor2.fg} The functions $f(\mu)$ and $g(\mu)$ possess the following asymptotic expansions as $\mu\to0$:
\begin{equation}\label{2.fexp}
f(\mu)=\pi\log\frac1\mu+\mu+\beta-4\pi\mu^{1/4}e^{-2\pi/\sqrt\mu}(1+o_\mu(1))
\end{equation}
and
\begin{equation}\label{2.gexp}
g(\mu)=\pi\log\frac1\mu+2\mu+\beta-\pi-4\pi^2\mu^{-1/4}e^{-2\pi/\sqrt\mu}(1+o_\mu(1)),
\end{equation}
where the integration constant $\beta=\pi(2\gamma+2\log2+3\log\pi-4\log\Gamma(1/4))$.
\end{corollary}
\begin{proof} Up to the integration constant $\beta$, expansions \eqref{2.fexp} and \eqref{2.gexp} are straightforward corollaries of \eqref{2.hexp} and \eqref{2.dif}, so we only mention here the explicit expression for the integral of the leading exponential term in \eqref{2.hexp} with respect to $\mu$:
$$
\int^\infty_\mu 4\pi^2x^{-5/4}e^{-2\pi/\sqrt x}\,dx=4\pi^2\sqrt2\erf(\sqrt{2\pi}\mu^{-1/4}),
$$
where $\erf(x)$ is the usual probability integral. Then, using the well-known expansions for the $\erf(z)$ near $z=\infty$, we find the leading exponential term in \eqref{2.fexp} (and \eqref{2.gexp} follows immediately from the second formula of \eqref{2.dif}).
\par
Thus, we only need to find the integration constant $\beta$. This, however, is
a much more delicate problem and the arguments above do not indicate how to
compute it. The derivation, based on the Hardy formula
for the 2D analogue of the Riemann zeta function, is given instead in the Appendix.
Corollary \ref{Cor2.fg} is proved.
\end{proof}
\begin{corollary}\label{Cor2.log} Let the functions $\Theta(\delta)$ and $\Theta_0(\delta)$ be defined via \eqref{1.theta} and \eqref{0.th} respectively. Then
\begin{equation}\label{2.0tt}
\Theta(\delta)=\Theta_0(\delta)+O(e^{-(2-\eb)\pi\delta^{1/2}})
\end{equation}
as $\delta\to\infty$ (here $\eb>0$ is arbitrary).
\end{corollary}
Indeed, \eqref{2.0tt} follows in a straightforward way from Theorem \ref{Th1.ext} and expansions \eqref{2.hexp}, \eqref{2.fexp} and \eqref{2.gexp} (even without the leading exponentially small terms).
\par
We are now want to check that $\Theta_0(\delta)$ is always {\it larger} than $\Theta(\delta)$.
We start by checking this property for large $\delta$.
\begin{lemma}\label{Lem2.larger} There exists $\delta_0>0$ such that
\begin{equation}\label{2.larger}
\Theta(\delta)\le \Theta_0(\delta)
\end{equation}
for all $\delta>\delta_0$.
\end{lemma}
\begin{proof} Let us introduce the following exponentially-corrected analogue of the
function $\Theta_0$:
\begin{multline}\label{2.th-exp}
\Theta_{\mbox{\scriptsize{exp}}}(\mu)=\frac1{4\pi^2}\cdot \frac{(\pi\log\frac1\mu+\beta+\mu-4\pi\mu^{1/4}e^{-2\pi/\sqrt\mu})^2}
{\pi\log\frac1\mu+\beta-\pi+2\mu-4\pi^2\mu^{-1/4}e^{-2\pi/\sqrt\mu}},\\ \delta(\mu)=\frac{\frac\pi\mu-1+4\pi^2\mu^{-5/4}e^{-2\pi/\sqrt{\mu}}}
{\pi\log\frac1\mu+\beta-\pi+2\mu-4\pi^2\mu^{-1/4}e^{-2\pi/\sqrt\mu}}.
\end{multline}
Then, according to \eqref{2.hexp}, \eqref{2.fexp} and \eqref{2.gexp}, the function
$\Theta_{\mbox{\scriptsize{exp}}}(\delta)$ gives a
{\it better} approximation to $\Theta(\delta)$ than $\Theta_0(\delta)$ if $\delta$ is large. Consequently, to prove the lemma,
it is sufficient to verify that
\begin{equation}\label{2.0exp}
\Theta_{\mbox{\scriptsize{exp}}}(\delta)\le \Theta_0(\delta)
\end{equation}
for large $\delta$. To this end, we introduce small $\eb:=4\pi\mu^{-1/4}e^{-2\pi/\sqrt\mu}$ and write
 \eqref{2.th-exp} in the form
\begin{equation}\label{2.th-eps}
\Theta(\mu,\eb)=\frac1{4\pi^2}\cdot \frac{(\pi\log\frac1\mu+\beta+\mu-\mu^{1/2}\eb)^2}
{\pi\log\frac1\mu+\beta-\pi+2\mu-\pi\eb},\ \ \delta(\mu,\eb)=\frac{\frac\pi\mu-1+\pi\mu^{-1}\eb}
{\pi\log\frac1\mu+\beta-\pi+2\mu-\pi\eb}.
\end{equation}
Then, since $\eb$ is extremely small in comparison with $\mu$ if $\mu$ is small, we may consider it as an infinitesimal increment. Therefore, \eqref{2.0exp} will be satisfied if and only if the infinitesimal shift along the vector $(\partial_\eb\delta,\partial_\eb\Theta)\big|_{\eb=0}$ lies {\it under} the tangent line to $(\delta(\mu,0),\Theta(\mu,0))$. This requires
us to verify the following condition:
$$
(\partial_\mu\Theta\cdot\partial_\eb\delta-\partial_\eb\Theta\cdot\partial_\mu\delta)\big|_{\eb=0}<0.
$$
Direct calculation gives
$$
(\partial_\mu\Theta\cdot\partial_\eb\delta-\partial_\eb\Theta\cdot\partial_\mu\delta)\big|_{\eb=0}=-
\frac1{2\pi^2}\frac{(\pi\log\frac1\mu+\beta+\mu)\cdot(\pi^2\log\frac1\mu+\pi\beta-2\pi^2+5\pi\mu-2\mu^2)}
{\mu^{3/2}(\pi\log\frac1\mu+\beta-\pi+2\mu)^3}
$$
and we see that the right-hand side is indeed negative if $\mu$ is small enough. Lemma \ref{Lem2.larger} is proved.
\end{proof}
Thus, the desired inequality \eqref{2.0tt} is analytically verified for large $\delta$. In contrast,
it is unlikely that it can be analogously checked for small values of $\delta$ since the asymptotic expansions do not work here and we need to work directly with the lattice sums. However, the {\it numerics} is reliable for $\delta$ `not large', so instead we check it numerically in that region.
As follows from our numerical simulations,
the conjecture is indeed true for all values of $\delta\ge1$. Thus, we have verified the validity of the following improved version of the logarithmic Sobolev inequality:
\begin{equation}\label{2.best}
\|u\|_{C(\Bbb T^2)}^2\le\|\nabla u\|^2_{L^2(\Bbb T^2)}\Theta_0\(\frac{\|\Delta u\|^2_{L^2(\Bbb T^2)}}{\|\nabla u\|^2_{L^2(\Bbb T^2)}}\)
\end{equation}
for all $2\pi\times2\pi$-periodic functions with zero mean.
\begin{remark}\label{Rem2.num} Note that, although the right-hand side of \eqref{2.best} does not contain any numerically-found constants, our verification of it is based on a combination of the analytic methods with numerics (which was used to check inequality \eqref{2.larger} for $\delta$ not large).
In fact, we do not present here a computer assisted proof, so being pedantic,
inequality \eqref{2.best} remains not rigorously proved. Nevertheless, error analysis for
our numerics could be done and it is likely that the foregoing could, in principle, be
upgraded to a rigorous computer proof.
\end{remark}
\begin{remark}\label{Rem2.num1} As we have already mentioned, the value of $\Theta(\delta)$ is extremely close to $\Theta_0(\delta)$ even for relatively small $\delta$ (e.g, for $\delta=4$, the difference is already less than $10^{-5}$), so the high precision computations of the lattice sums \eqref{2.sums} are required in order to show that $\Theta$ is indeed smaller than $\Theta_0$. Thus, the direct computations of that sums require to count very many terms and rather slow. Alternatively, using the Poisson summation formula, we have
$$
h(\mu)=\frac\pi\mu-1+2\pi^2\mu^{-3/2}\Sum |k|\cdot K_1(2\pi\mu^{-1/2}|k|)
$$
and, integrating this over $\mu$ and keeping in mind the value of the integration constant, we arrive at
$$
f(\mu)=\pi\log\frac1\mu+\beta+\mu-8\pi\Sum K_0(2\pi\mu^{-1/2}|k|),\ \ g(\mu)=f(\mu)-\mu h(\mu).
$$
Since the Bessel functions $K_0(x)$ and $K_1(x)$ decay exponentially as $|x|\to\infty$, the transformed series converge much faster and look preferable for the high precision computations. Actually, we use both of that approaches in order to double check our numerics.
\end{remark}
Our next task is to present rougher version of \eqref{2.best}, approximating the right-hand
side of \eqref{2.best} by simpler functions. To this end, we need the following lemma.
\begin{lemma}\label{Lem2.loglog} The function $\Theta(\delta)$ possesses the following asymptotic expansion:
\begin{equation}\label{2.loglog}
\Theta(\delta)=\frac1{4\pi}\(\log\delta+\log\log\delta+\frac{\beta+\pi}\pi+\frac{\log\log\delta}{\log\delta}+O((\log\delta)^{-1})\)
\end{equation}
as $\delta\to\infty$.
\end{lemma}
\begin{proof} We first note that, since $\Theta(\delta)$ is {\it exponentially} close to $\Theta_0(\delta)$, we may verify \eqref{2.loglog} for the function $\Theta_0$ only. To this end, we need to find the expansion for $\mu=\mu(\delta)$ from
$$
\delta=\frac{\frac\pi\mu-1}{\pi\log\frac1\mu+\beta-\pi+2\mu}
$$
and insert it into the expression for $\Theta_0(\mu)$. To compute the expansion
for $\mu(\delta)$, we drop out the term $-2\mu$
in the denominator (which only leads to an error of order $O(\delta^{-1+\eb})$, $\eb>0$, in the final answer). Then, the equation obtained
$$
\frac{\frac\pi\mu-1}{\pi\log\frac1\mu+\beta-\pi+2\mu}=\delta
$$
can be solved explicitly in terms of the so-called Lambert W-function:
\begin{equation}\label{2.W}
\frac1\mu=-\delta W_{-1}\(-\delta^{-1}e^{-\frac{1+\delta(\beta-\pi)}{\pi\delta}}\),
\end{equation}
where $W_{-1}$ is the $-1$-branch of the Lambert function (see \cite{Lam} for details). We
also note that
\begin{equation}\label{2.simple}
\Theta_0(\mu)=\frac1{4\pi}\log\frac1\mu+\frac{\beta+\pi}{4\pi^2}+O((\log\frac1\mu)^{-1})
\end{equation}
and, therefore, the remainder is again non-essential for \eqref{2.loglog} and can be dropped out. Thus, it remains
only to expand the logarithm of the right-hand side of \eqref{2.W}. To this end, we use the
expansion (see~\cite{Lam}) for the Lambert function $W_{-1}$ near zero:
\begin{equation}\label{2.Lambert}
W_{-1}(-z)=\log z-\log(-\log z)+O(\frac{\log(-\log z)}{\log(-z)}),\  \ z\to0-.
\end{equation}
This gives
$$
\frac1\mu=\delta\log\delta+\delta(\pi-\beta)-\frac1\pi+
\delta\log(\log\delta+\pi-\beta-\frac1{\pi\delta})+O(\frac{\log\log\delta}{\log\delta}).
$$
Taking the logarithm of the right-hand side of this formula, inserting the result in \eqref{2.simple} and dropping out the lower order terms, we end up with \eqref{2.loglog} and finish the proof of the lemma.
\end{proof}
We are now ready to state the improved logarithmic Sobolev inequality with
double-logarith\-mic correction.
\begin{theorem}\label{Th2.loglog} The following inequality holds
\begin{equation}\label{2.in-loglog}
\|u\|^2_{C(\Bbb T^2)}\le \frac1{4\pi}\|\nabla u\|^2_{L^2(\Bbb T^2)}\(\log\frac{\|\Delta u\|^2_{L^2(\Bbb T^2)}}{\|\nabla u\|^2_{L^2(\Bbb T^2)}}+\log\(1+\log\frac{\|\Delta u\|^2_{L^2(\Bbb T^2)}}{\|\nabla u\|^2_{L^2(\Bbb T^2)}}\)+L\)
\end{equation}
for all $2\pi\times2\pi$-periodic functions $u$ with zero mean. The constant $L>\frac{\beta+\pi}\pi$ is defined as follows
\begin{equation}\label{2.const}
L:=\max_{\delta\ge1}\bigg\{4\pi\Theta(\delta)-(\log\delta+\log(1+\log\delta))\bigg\}.
\end{equation}
This maximum is achieved at some finite $1<\delta_*<\infty$ and the corresponding conditional extremal $u_{\mu(\delta_*)}(x)$ is an exact extremal function for \eqref{2.in-loglog}.
\end{theorem}
\begin{proof} In the light of the asymptotic expansion \eqref{2.loglog} and the fact that $\Theta(\delta)$ is continuous, the supremum over $\delta\ge1$ of the function on the right-hand side of \eqref{2.const} is finite. Moreover, since the first decaying term in that expansion,
($\log\log\delta/\log\delta$), is positive, the inequality cannot
hold with $L=\frac{\beta+\pi}\pi$.
In a fact, there is an extra `1' in the double-logarithmic term in \eqref{2.in-loglog} in comparison with \eqref{2.loglog}, introduced in order that
the right hand side be a well-defined function for all $\delta\ge1$.
However, this term is only an $O\left((\log\delta)^{-1}\right)$ correction, which
is weaker than the first decaying term in the expansions \eqref{2.loglog} and cannot change anything.
\par
Thus, the above supremum cannot achieved as $\delta\to\infty$ and, therefore, since $\Theta(\delta)$ is continuous, it is achieved at some finite point $\delta=\delta_*$ and must be larger than the value at infinity ($L>\frac{\beta+\pi}\pi
\sim 1.82283$). Then, by the definition of $\Theta$ and $L$, inequality \eqref{2.in-loglog} holds and equality is achieved on the function $u(x)=u_{\mu(\delta_*)}(x)$. Theorem \ref{Th2.loglog} is proved.
\end{proof}
According to our numerical analysis, the maximum in \eqref{2.const} is {\it unique} and
is achieved at $\delta_*\sim 3.92888$ which corresponds to $L\sim 2.15627$. Thus,
the exact extremum function $u_{\mu(\delta_*)}(x)$ is also {\it unique} up to
translations, scaling and alternation.
\par
We conclude this section by analysing the structure of the extremal functions $u_\mu(x)$ for small,
positive $\mu$ (corresponding to large $\delta$).
\begin{lemma}\label{Lem2.PDE} The extremals $u_\mu(x)$ possess the following expansions:
\begin{equation}\label{2.PDEext}
u_\mu(x)=-2\pi K_0\(\mu^{-1/2}|x|\)+ G_0(x)+\mu+C_\mu+V_\mu(x),
\end{equation}
where $G_0(x)$ is a fundamental solution of the Laplacian:
\begin{equation}\label{2.lap}
\Dx G_0=-4\pi^2\delta(x)+1, \ \partial_nG_0\big|_{\partial([-\pi,\pi]^2)}=0,\ \ \int_{\Bbb T^2}G_0(x)\,dx=0;
\end{equation}
$C_\mu:=-\frac\mu{2\pi}\int_{\R^2\backslash \mu^{-1/2}[-\pi,\pi]^2}K_0(x)\,dx$ is an exponentially small
(with respect to $\mu\to0+$) constant;
the exponentially small function $V_\mu(x)$ solves the following fourth order
elliptic equation in $T=[-\pi,\pi]^2$ with non-homogeneous boundary conditions:
\begin{multline}\label{2.helm}
\Delta(1-\mu\Delta)V_\mu=0,\ \partial_n V_\mu\big|_{\partial T}=2\pi\partial_n K_0(\mu^{-1/2}|x|)\big|_{\partial_n T},\\
\partial_n\Dx V_\mu\big|_{\partial T}=2\pi\partial_n \Dx K_0(\mu^{-1/2}|x|)\big|_{\partial_n T},\ \ \int_{\Bbb T^2} V_\mu(x)\,dx=0;
\end{multline}
and $K_0$ is the zero-order Bessel K function.
\end{lemma}
\begin{proof} According to Remark \ref{Rem1.PDE}, the function $V_\mu(x)$ solves the fourth order elliptic equation \eqref{1.PDE} with periodic boundary conditions. Moreover, owing
to the symmetry, the periodic boundary conditions can be replaced by homogeneous Neumann ones.
Now let $G_0$ be the fundamental solution of the Laplacian in a square defined by \eqref{2.lap}
(the solution of this equation exists since the right-hand side has zero mean). Then, using the fact that
\begin{equation}\label{2.delta}
(1-\mu\Dx)K_0(\mu^{-1/2}|x|)=+2\pi\mu\delta(x)
\end{equation}
we end up with
$$
\Dx(1-\Dx)[2\pi K_0(\mu^{-1/2}|x|)+G_0(x))=-4\pi^2\delta(x)+1
$$
and, therefore, using also the obvious fact that $\partial_n\Dx G_0(x)\big|_{\partial T}=0$, we see that the remainder $V_\mu$ should indeed satisfy \eqref{2.helm}. The solvability condition
$$
(1-\mu\Dx)K_0(\mu^{-1/2}|x|)\big|_{\partial T}=0
$$
for that equation is satisfied in the light of \eqref{2.delta}. Thus, the decomposition \eqref{2.PDEext} is verified up to a constant (we recall that the function $u_\mu(x)$ must have zero mean). In order to find this constant, we note that, by definition, the functions
$G_0(x)$ and $V_\mu(x)$ have zero means, so only the function $K_0$ has non-zero mean and
hence the constant is determined by
\begin{multline*}
C=\frac1{2\pi}\int_{\Bbb T^2}K_0(\mu^{-1/2}|x|)\,dx=-\frac\mu{2\pi}\int_{\mu^{-1/2}\Bbb T^2}K_0(|x|)\,dx=\\=
\frac\mu{2\pi}\(\int_{\R^2}K_0(|x|)\,dx-\int_{\R^2\backslash \mu^{-1/2}T}K_0(|x|)\,dx\)=\mu+C_\mu.
\end{multline*}
The constant $C_\mu$ is indeed exponentially small as $\mu\to0+$ since the function $K_0(z)$ is exponentially decaying as $z\to\infty$.
\end{proof}
Recall that
$$
G_0(x)=2\pi\log\frac1{|x|}+\text{`smooth remainder'};
$$
therefore, the leading term of $u_\mu(x)$ up to smooth zero order terms in $\mu$ is radially symmetric and is given by
\begin{equation}\label{2.sym}
u_\mu(x)=2\pi\(\log\frac1{|x|}-K_0(\mu^{-1/2}|x|)\)+\text{`smooth, order zero remainder'}.
\end{equation}
Thus, $u_\mu(x)$ consists of a radially symmetric spike near $x=0$ corrected by lower
order terms. Figure~\ref{cont_plot} shows a contour plot of $u_\mu(x)$ for
$\mu\approx 0.12211$, which corresponds to $\delta = \delta_*$ (see Theorem \ref{Th2.loglog}).


\begin{figure}[htbp]
\centering
\includegraphics*[angle=0,width=4.5in]{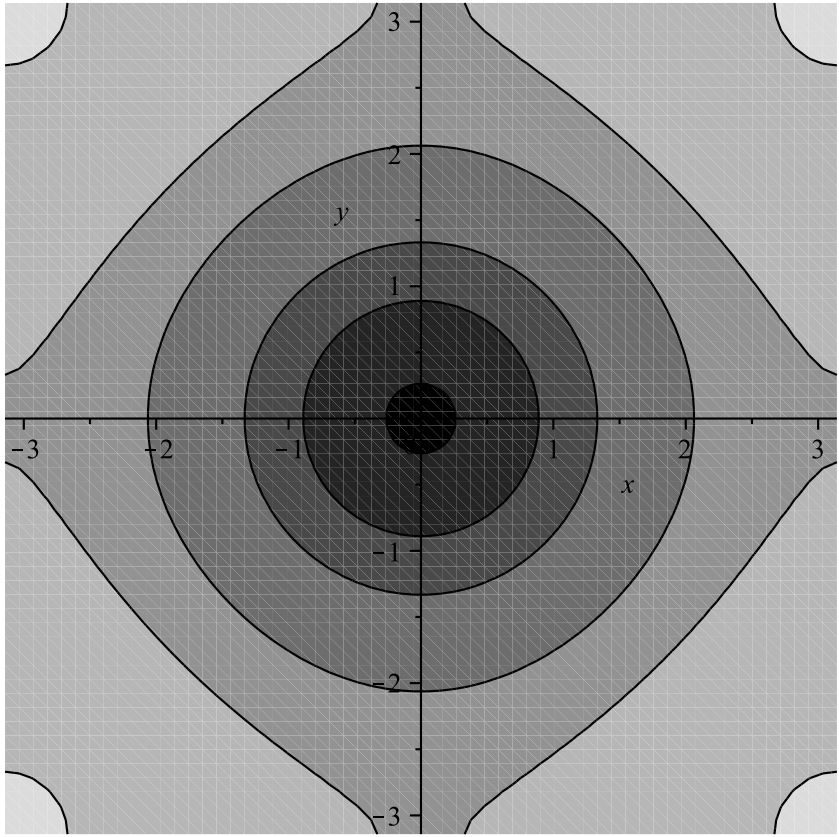}
\caption{A contour plot of $u_\mu(x, y)$ for $\delta = \delta_*\approx
3.92888$ ($\mu\approx 0.12211$). Darker areas are higher. The spike becomes
almost perfectly radially symmetric, even for this relatively small value of $\delta$.}
\label{cont_plot}
\end{figure}


\begin{remark}\label{Rem2.connections} Passing to the limit $\mu\to0$ (for $|x|\ne0$) in \eqref{2.lap} and using the lattice sum formula \eqref{0.extmu} for the extremals, we see that (at least formally)
\begin{equation}\label{2.lattice}
G_0(x)=\Sum \frac{e^{ik\cdot x}}{k^2}.
\end{equation}
It can be shown that the sign alternating sum on the right-hand side is convergent (if the
proper order of summation is chosen) for every $x\ne0$, and the equality holds; see \cite{BB,Gla}.
In addition, using the known asymptotic expansion for the Bessel K-function near zero
$$
K_0(z)=-\log z+\log2 -\gamma+O(z^2),
$$
see \cite{Wat}, together with \eqref{2.PDEext} and \eqref{2.fexp}, one can show that the integration constant $\beta$ can be
expressed in terms of $G_0$ as follows:
\begin{multline}\label{2.bPDE}
\beta=2\pi\gamma-2\pi\log2+\lim_{x\to0}\(G_0(x)-2\pi\log\frac1{|x|}\)=\\=2\pi(\gamma-\log2)+\lim_{x\to0}\(\Sum \frac{e^{ik\cdot x}}{k^2}-2\pi\log\frac1{|x|}\).
\end{multline}
Recall also that the fundamental solution $G_0(x)$ can be explicitly written in terms of
integrals of some elliptic functions (e.g., using the bi-conformal map between the square
and the unit circle) and the values of $G_0(x)$ can be explicitly found for some $x$ by using
identities for elliptic functions; for instance,
$$
G_0((\pi,\pi))=\sum_{(k_1,k_2)\in\Bbb Z^2-\{0\}}\frac{(-1)^{k_1+k_2}}{k_1^2+k_2^2}=-\pi\log2,
$$
see \cite{Fi,Gla}. However, we have failed to find the limit \eqref{2.bPDE} in this way, so
our computation of the integration constant $\beta$ (see the Appendix) will be based on
different arguments.
\end{remark}

\section{Elementary approaches to the logarithmic Sobolev inequality}\label{s3}

In this section, we discuss the possibility of obtaining inequality \eqref{2.in-loglog} with {\it sharp} constant $\frac1{4\pi}$ (at least in the leading term $\log\delta$) using the standard strategies for proving the logarithmic Sobolev inequality. In a fact, we will analyse two such strategies. The first one  is based on the embedding of $H^{1+\eb}$ to $C$ for every $\eb>0$:
\begin{equation}\label{3.emb}
\|u\|_{C(\Bbb T^2)}^2\le \frac{C}\eb\|u\|^2_{H^{1+\eb}},
\end{equation}
where $C$ is independent of $\eb\to0$, the interpolation $\|u\|_{H^{1+\eb}}\le C\|u\|_{H^1}^{1-\eb}\|u\|_{H^2}$ and the proper choice of $\eb$ ($\eb\sim(\log\delta)^{-1})$.
\par
The second strategy consists of splitting the function $u$ into lower and higher Fourier modes
\begin{equation}\label{3.sharpp}
u(x)=\Sum u_k e^{ik\cdot x}=\Sum_{|k|\le N} u_k e^{ik\cdot x}+\Sum_{|k|> N} u_k e^{ik\cdot x}
\end{equation}
with a properly chosen $N\sim\delta$, and estimating the lower and higher Fourier modes using the $H^1$ and $H^2$-norms respectively.
\par
As we will see, the first scheme is {\it rough} and can give only the $e$-times larger constant $\frac{e}{4\pi}$ in the leading term (even if the best constants in the intermediate inequalities
are chosen). By contrast, the second scheme is much sharper and allows correct
retrieval not only of the leading term, but also of the double-logarithmic correction.
\par
We start with the first approach (following to \cite{rubbish}). To proceed, we first need the sharp constant in the $L^\infty$-embedding \eqref{3.emb}.
\begin{lemma}\label{Lem3.sharp} Let $\eb>0$ be arbitrary. Then, for every $u\in H^{1+\eb}(\Bbb T^2)$ with zero mean, the following inequality holds:
\begin{equation}\label{3.good}
\|u\|^2_{C(\Bbb T^2)}\le C(\eb)\|(-\Dx)^{\frac{1+\eb}2}u\|^2_{L^2(\Bbb T^2)},\ \ C(\eb):=\frac1{4\pi^2}\Sum\frac1{|k|^{2(1+\eb)}}.
\end{equation}
The constant $C(\eb)=\frac1{4\pi\eb}+O_{\eb\to0}(1)$ is sharp and the exact extremals are given by
\begin{equation}\label{3.ext}
U_\eb(x):=\Sum \frac{e^{ik\cdot x}}{|k|^{2(1+\eb)}}
\end{equation}
(up to scalings and shifts).
\end{lemma}
\begin{proof} Indeed,
\begin{multline*}
\|u\|_{L^\infty}^2\le\(\Sum |u_k|\)^2=\(\frac1{|k|^{1+\eb}}\cdot(|k|^{1+\eb}|u_k|)\)^2\le\\\le\Sum\frac1{|k|^{2(1+\eb)}}\cdot\Sum |k|^{2(1+\eb)}|u_k|^2=C(\eb)\|(-\Delta)^{\frac{1+\eb}2}u\|^2_{L^2}
\end{multline*}
and the equalities here hold if  $u_k=C\frac1{|k|^{2(1+\eb)}}$, which
gives \eqref{3.ext}.
\par
The leading term in the asymptotic expansions of $C(\eb)$ can easily be found, say, by replacing
the sum with the corresponding integrals (see Lemma \ref{LemA.1} in Appendix).
\end{proof}
\begin{remark}\label{Rem3.Hardy} The lattice sum for $C(\eb)$ can be computed in a closed form through the Riemann zeta and Dirichlet beta functions using the Hardy formula:
\begin{equation}\label{3.hardy}
\Sum\frac1{|k|^{2(1+\eb)}}=4\zeta(1+\eb)\beta(1+\eb),\ \ \beta(z):=\sum_{n=0}^\infty\frac{(-1)^{n}}{(2n+1)^z},
\end{equation}
see \cite{Zu}.
This formula, together with the asymptotic expansions of $\zeta(1+\eb)$ and $\beta(1+\eb)$, will be
required in the Appendix in order to compute the integration constant $\beta$.
\end{remark}
We now recall that the sharp constant in the interpolation inequality
\begin{equation}\label{3.int}
\|(-\Delta)^{\frac{1+\eb}2}u\|_{L^2}\le\|\nabla u\|_{L^2}^{1-\eb}\|\Dx u\|^\eb_{L^2}
\end{equation}
is unity and the exact extremals are the eigenfunctions of the Laplacian:
\begin{equation}\label{3.lapext}
U_k(x)=e^{ik\cdot x},\ \ k\in\Bbb Z^2-\{0\};
\end{equation}
see \cite{Tri} for the details. Thus, combining \eqref{3.good} and \eqref{3.int}, we may write
\begin{multline}\label{3.bad}
\|u\|_{C(\Bbb T^2)}^2\le \inf_{\eb\in(0,1]}\bigg\{C(\eb)\|\nabla u\|^{2(1-\eb)}_{L^2}\|\Delta u\|^\eb_{L^2}\bigg\}=\|\nabla u\|^2
\inf_{\eb\in(0,1]}\{C(\eb)\delta^\eb\}=\\ \frac1{4\pi}\|\nabla u\|^2_{L^2}\min_{\eb\in(0,1]}\bigg\{e^{\eb\log\delta}(\eb^{-1}+O_{\eb\to0}(1))\bigg\}=\frac1{4\pi}\|\nabla u\|^2_{L^2}(e\log\delta+O_{\delta\to\infty}(1)).
\end{multline}
(the last minimum being achieved for $\eb\sim(\log\delta)^{-1}$ if $\delta$ is large). Thus,
the above described approach {\it is not sharp} and  gives an $e$-times larger constant for
the leading term on the right-hand side of the inequality considered.
\par
This result is not, in fact, surprising if we compare the extremals $u_\mu(x)$ for the
logarithmic Sobolev inequality with the extremals \eqref{3.lapext} for the interpolation
inequality used in the above arguments. Indeed the first ones are delta-like spikes situated
near zero, but the others are well-distributed rapidly oscillating functions. Thus, we are
applying the interpolation inequality to
functions which are very far from the extremals and for this reason we may expect that on the
extremals, this inequality holds with constant better than unity (see \cite{C}).
\par
By contrast, the extremals \eqref{3.ext} look very similar to $u_\mu(x)$: both of them are delta-like spikes with height proportional to $\log\frac1\mu$ (if we take the optimal $\eb\sim\log\frac1\mu$). Therefore, one may expect that the sharpness of the above scheme is lost mainly due to usage of the
interpolation and that it is probably possible to retrieve the sharp constant by using only the
first inequality \eqref{3.good}:
\begin{equation}\label{3.stupid}
\|u\|^2_{C(\Bbb T^2)}\le \inf_{\eb\in(0,1]}\bigg\{C(\eb)\|(-\Delta)^{\frac{1+\eb}2}u\|_{L^2}^2\bigg\}
\end{equation}
and then computing the infimum in the right-hand side in some `more clever' way.
\par
However, surprisingly, this expectation is {\it wrong} and approximately the same
`degree of sharpness' is lost under the usage
of the first \eqref{3.good} and the second \eqref{3.int} inequalities. In order to see this,
we compute the leading terms of the asymptotic expansions in $\mu$ for the left and right-hand sides of \eqref{3.stupid} on the conditional extremals $u_\mu(x)$ of the logarithmic Sobolev
inequality considered.

 \begin{lemma}\label{Lem3.killstupid} Let
 \begin{equation}\label{3.AB}
 A(\mu):=\|u_\mu\|_{C(\Bbb T^2)}^2,\ \ B(\mu):=\inf_{\eb\in(0,1]}\bigg\{C(\eb)\|(-\Delta)^{\frac{1+\eb}2}u_\mu\|_{L^2}^2\bigg\},
 \end{equation}
 where the functions $u_\mu(x)$ are given by \eqref{0.extmu}. Then, the following expansions hold:
 \begin{equation}\label{e.ABexp}
 A(\mu)=\pi^2\log^2\frac1\mu+O(\log\frac1\mu),\ \ B(\mu)=\pi^2\alpha\log^2\frac1\mu+O(\log\frac1\mu)
 \end{equation}
 as $\mu\to0+$. The constant $\alpha>1$ is given by
\begin{equation}\label{3.al}
 \alpha:=\frac{e^{W(-2\exp(-2))+2}-1}{(W(-2\exp(-2))+2)^2}\sim1.544,
\end{equation}
 where $W(z)$ is the principal branch of the Lambert's $W$-function, see \cite{Lam}.
\end{lemma}
\begin{proof} The asymptotic expansion for the function $A(\mu)$ follows from \eqref{2.fexp} and we only need to study the function $B(\mu)$. To this end, we introduce a function
\begin{equation}\label{3.mueps}
h(\mu,\eb):=\frac1{4\pi^2}\|(-\Dx)^{\frac{1+\eb}2}u_\mu\|^2_{L^2}=\Sum\frac1{|k|^{2(1-\eb)}(1+\mu k^2)}.
\end{equation}
Note that the infimum on the right-hand side of \eqref{3.AB} is achieved for {\it small} $\eb$ when $\mu$ is small. We may assume, without loss of generality, that $\eb<1/2$. Then, applying
estimate \eqref{13} (see Appendix) we see that the one-dimensional integrals are uniformly
bounded as $\eb\to0$ and $\mu\to0$ and we may write
\begin{multline*}
f(\mu,\eb)=\int_{|x|>1}\frac{dx}{|x|^{2(1-\eb)}(1+\mu|x|^2}+O_{\mu,\eb}(1)=
2\pi\mu^{-\eb}\int_{r\ge\mu^{1/2}}\frac{dr}{r^{2(1-\eb)}(1+r^2)}+O_{\mu,\eb}(1)=\\=2\pi\mu^{\eb}
\(\int_0^\infty\frac{dr}{r^{2(1-\eb)}(1+r^2)}-\int_{r\le\mu^{1/2}}\frac{dr}{r^{2(1-\eb)}(1+r^2)}\)+O_{\mu,\eb}(1)=\\=
2\pi\mu^{-\eb}\(\frac{\pi}{2\sin(\pi\eb)}-\frac{\mu^\eb}\eb\)+O_{\mu,\delta}(1)=
\pi\frac{\mu^{-\eb}\cdot\frac{\pi\eb}{\sin(\pi\eb)}-1}{\eb}+O_{\mu,\eb}(1),
\end{multline*}
where we have used the fact that the first integral in the middle line can be found explicitly
and the second one can be computed up to the bounded terms using the expansions $\frac1{1+x^2}=1+\sum_{n=1}^\infty(-1)^nx^{2n}$.
\par
Recalling now that $C(\eb)=\frac1{4\pi\eb}+O_\eb(1)$, we end up with
$$
B(\mu)=\pi^2\max_{\eb\in(0,1]}\left\{\(\frac1\eb+O_\eb(1)\)
\(\frac{\mu^{-\eb}\cdot\frac{\pi\eb}{\sin(\pi\eb)}-1}{\eb}+O_{\mu,\eb}(1)\)\right\}.
$$
It is not difficult to see that the leading term as $\mu\to0$ in the minimising problem is given by
\begin{equation}\label{3.min}
\pi^2\min_{\eb\in[0,1]}\left\{\frac{\mu^{-\eb}-1}{\eb^2}\right\}=\pi^2\min_{\eb\in[0,1]}\left\{\frac{e^{\eb\log(\mu^{-1})}-1}{\eb^2}\right\}
=\pi^2\log^2\frac1\mu\,\cdot\, \min_{\gamma>0}\left\{\frac{e^\gamma-1}{\gamma^2}\right\}
\end{equation}
and the remainder term will be of the order $\log\frac1\mu$ as $\mu\to0$. It only remains to note that the minimum on the right-hand side of \eqref{3.min} can be found explicitly in terms of the Lambert W-function and coincides with \eqref{3.al}. Lemma \ref{Lem3.killstupid} is proved.
\end{proof}
\begin{remark}\label{Rem3.rough} Thus, since the right-hand side of \eqref{2.in-loglog} computed on the extremals $u_\mu(x)$ gives the same leading term in the asymptotic expansions as the function $A(\mu)$, we see that
is {\it impossible} to obtain \eqref{2.in-loglog} with constant better than
$\frac\alpha{4\pi}$ if inequality \eqref{3.stupid} is used (no matter how sharply we
further estimate the right-hand side of \eqref{3.stupid}).
\end{remark}
We now return to the second of the methods described above. To this end, we estimate the first
and the second term on the right-hand side of \eqref{3.sharpp} as follows:
\begin{multline*}
 \Sum_{|k|\le N}|u_k|=\Sum_{|k|\le N}k^{-1}\cdot k|u_k|\le\\\le \(\Sum_{|k|\le N}|k|^{-2}\)^{1/2}\(\Sum_{|k|\le N} k^2|u_k|^2\)^{1/2}\le
 \frac1{2\pi}\|\nabla u\|_{L^2}\(\Sum_{|k|\le N}|k|^{-2}\)^{1/2}
\end{multline*}
and
\begin{multline*}
 \Sum_{|k|> N}|u_k|=\Sum_{|k|> N}|k|^{-2}\cdot k^2|u_k|\le\\\le \(\Sum_{|k|> N}|k|^{-4}\)^{1/2}\(\Sum_{|k|> N} k^2|u_k|^2\)^{1/2}\le
 \frac1{2\pi}\|\Delta u\|_{L^2}\(\Sum_{|k|> N}|k|^{-4}\)^{1/2}
\end{multline*}
which together with \eqref{3.sharpp} leads to the following estimate:
\begin{equation}\label{3.notstupid}
\|u\|_{C(\Bbb T^2)}^2\le \frac1{4\pi^2}\|\nabla u\|^2\min_{N>0}\(\(\Sum_{|k|\le N}\frac1{|k|^{2}}\)^{1/2}+\delta^{1/2}\(\Sum_{|k|> N}\frac1{|k|^{4}}\)^{1/2}\)^2
\end{equation}
with $\delta=\frac{\|\Delta u\|^2_{L^2}}{\|\nabla u\|^2_{L^2}}$. The following lemma gives the asymptotic behaviour of the right-hand side of this inequality as $\delta\to\infty$.
\begin{lemma}\label{Lem3.strange} Let $P(\delta)$ be the value of the minimum on the right-hand side of \eqref{3.notstupid}. Then this function possesses the following expansion:
\begin{equation}\label{3.P}
P(\delta)=\frac1{4\pi}\(\log\delta+\log\log\delta+\frac{\beta+\pi}\pi\)+\frac{1+\log2}{4\pi}+o(1)
\end{equation}
as $\delta\to\infty$.
\end{lemma}
\begin{proof} As shown in the Appendix (see Lemma \ref{LemA.3}),
$$
\Sum_{|k|\le N}\frac1{|k|^2}=2\pi\log N+\beta+O(N^{-1})
$$
as $N\to\infty$. On the other hand, as is not difficult to show, using,
say, Lemma \ref{LemA.1},
\begin{equation}\label{3.high}
\Sum_{|k|>N}\frac1{|k|^4}=\int_{|x|>N}\frac{dx}{|x|^4}+O(N^{-3})=\pi N^{-2}+O(N^{-3}).
\end{equation}
Then, using the obvious fact that the minimum on the right-hand side of \eqref{3.notstupid} should be achieved for $N$ `close'
to $\delta$ ($C_\gamma^{-1}\delta^{1-\gamma}\le N_{min}\le C_\gamma\delta^{1+\gamma}$, for all $\gamma>0$), we see that
\begin{equation}\label{3.Wagain}
P(\delta)=\frac1{2\pi}\min_{N>0}\(\log^{1/2} (kN)+\frac1{\sqrt2}N^{-1}\delta^{1/2}\)^2+o(1),\ \ k:=e^{\beta/(2\pi)}
\end{equation}
as $\delta\to\infty$. Differentiating the expression on the right-hand side, we see that the minimum is achieved at
$$
N(\delta):=\frac{e^{-\frac12W_{-1}(\frac{-1}{k^2d})}}k,
$$
where $W_{-1}(z)$ is again the $-1$-branch of the Lambert $W$-function. Using the expansion \eqref{2.Lambert} for the Lambert $W$-function, we arrive at
$$
N_{min}(\delta)=\delta^{1/2}\sqrt{\log(k^2 d)}\(1+O\(\frac{\log\log\delta}{\log\delta}\)\).
$$
Inserting this expression into the right-hand side of \eqref{3.Wagain} we end up with \eqref{3.P} (after some straightforward computations) and finish the proof of the lemma.
\end{proof}
\begin{remark}\label{Rem3.comp} Thus, in contrast to the first method, the second
gives {\it two correct terms} in the asymptotic expansion of the function $\Theta(\delta)$ and the error appears only in the third term (wrong additional constant
$\frac{1+\log2}{4\pi}$, compare \eqref{0.th} and \eqref{3.P}) and we
conclude that the second method is sharper and clearly preferable for the elementary
proof of inequalities of this type, at least in the case of tori.
\end{remark}

\section{The algebraic case}\label{s4}
In this section, we apply the method developed above to the simpler case of {\it algebraic}
interpolation inequalities of the form \eqref{0.alg} on the torus. We are able to treat the
case of tori
of arbitrary dimension $d$; however, in order to avoid the computation of the analogues of the integration constant $\beta$ (which is difficult and requires more refined analysis), we restrict
ourselves to the case where one of the interpolation spaces is $L^2$. So, we want to analyse the
interpolation inequality
\begin{equation}\label{4.1d-new}
\|u\|_{C(\Bbb T^d)}^2\le c_d(n)\|u\|_{L^2}^{2-\frac dn}\|(-\Dx)^{n/2}\|^{\frac dn}_{L^2}
\end{equation}
for $(2\pi)^d$-periodic functions with zero mean.
Following the above described scheme, we replace inequality \eqref{4.1d-new} by the refined one
\begin{equation}\label{4.1dbest-new}
\|u\|^2_{C(\Bbb T^d)}\le \|u\|^2_{L^2}\Theta_{d,n}(\delta),\ \ \delta:=\frac{\|(-\Dx)^{n/2}\|_{L^2}^2}{\|u\|_{L^2}^2}\ge1,
\end{equation}
where
\begin{equation}\label{4.1dext-new}
\Theta_{d,n}(\delta):=\sup\bigg
\{\|u\|^2_{L^2},\ u\in H^n(\Bbb T^d),\ \|u\|_{L^2}=1,\ \|(-\Dx)^{n/2}\|_{L^2}^2=\delta, \int_{\Bbb T^d}u(x)\,dx=0\bigg\}.
\end{equation}
First of all we note that, arguing as in Lemma \ref{Lem1.exist}, we may prove that the
maximiser $u_\delta(x)$ for problem \eqref{4.1dext-new}
exists. So, we may apply the Lagrange multipliers technique, analogously to Theorem \ref{Th1.ext} and obtain the following result
\begin{lemma}\label{Lem4.exist-new} The conditional extremals for problem \eqref{4.1dbest-new} are given by
\begin{equation}\label{4.umu-new}
u_\mu(x)=\Sum\frac{e^{ik\cdot x}}{1+\mu |k|^{2n}},\ \ \mu\in(-\infty,-1]\cup(0,+\infty]
\end{equation}
(where $\Sum$ now means the sum over the lattice $k\in\Bbb Z^d$ except $k=0$)
and, therefore, the desired function $\Theta_{d,n}(\delta)$ possesses the following parametric description:
\begin{equation}\label{4.1ddtheta-new}
\Theta_{d,n}(\mu):=\frac1{(2\pi)^d}\frac{\(\Sum\frac1{1+\mu |k|^{2n}}\)^2}{\Sum\frac1{(1+\mu |k|^{2n})^2}},\ \ \delta(\mu)=\frac{\Sum\frac{|k|^{2n}}{1+\mu |k|^{2n}}}{\Sum\frac1{(1+\mu |k|^{2n})^2}},
\end{equation}
where $\mu\in(-\infty,-1]\cup(0,+\infty]$. In addition, for every $\delta\ge1$ there is a unique $\mu=\mu(\delta)$ belonging to that interval.
\end{lemma}
The proof of this lemma is analogous to Theorem \ref{Th1.ext} and, therefore, is omitted.
\par
Furthermore, analogously to \eqref{2.fexp}, \eqref{2.gexp} and \eqref{2.hexp}, we may find the asymptotic expansions up to exponential terms for all sums involving into the parametric definition of the function $\Theta_n$.
\begin{lemma}\label{Lem4.exp-new} Let $n\in\Bbb N$ and $2n-d>0$. Then, the following expansions hold:
\begin{multline}\label{4.fgh-new}
f(\mu):=\Sum\frac{1}{1+\mu |k|^{2n}}=\frac{\pi\omega(d)}{2n\sin(\frac{\pi d}{2n})}\cdot\mu^{\frac{-d}{2n}}-1+O(e^{\frac{-C_n}{\mu^{d/(2n)}}}),\\ g(\mu):=\Sum\frac{1}{(1+\mu |k|^{2n})^2}=\frac14\cdot\frac{\pi(2n-d)\omega(d)}{n^2\sin(\frac{\pi d}{2n})}
\cdot\mu^{\frac{-d}{2n}}-1+O(e^{\frac{-C_n}{\mu^{d/(2n)}}}),\\
h(\mu):=\Sum\frac{|k|^{2n}}{(1+\mu |k|^{2n})^2}=\frac14\cdot\frac{\pi d\omega(d)}{n^2\sin(\frac{\pi d}{2n})}
\cdot\mu^{-1-\frac{d}{2n}}+O(e^{\frac{-C_n}{\mu^{d/(2n)}}})
\end{multline}
as $\mu\to0$. Here $\omega(d):=\frac{2\pi^{d/2}}{\Gamma(d/2)}$ is the volume of the $(d-1)$-dimensional unit sphere and $C_n$ is a positive constant depending on $n$.
\end{lemma}
\begin{proof}
Expansions \eqref{4.fgh-new} can be obtained from the Poisson summation formula
analogously to Lemma \ref{Lem2.Poi}, but more simply since we do not need to analyse the leading
exponentially decaying term here, so we need not find the Fourier transforms explicitly and
may just use the fact that the full sums (including the term with $k=0$) are exponentially
close to the
corresponding integrals. Note also that, in contrast to Section \ref{s2}, we do not have
singularities at $k=0$ in any sums, so the additional integration constant does not appear
and the verification of \eqref{4.fgh-new} is reduced to computing the multi-dimensional
integrals associated with the sums. In turn, the integrals can be straightforwardly computed
using hyperspherical coordinates (recall that all the integrals are radially symmetric)
and the following well-known formulae:
\begin{equation}\label{4.BG-new}
\int_0^\infty\frac{x^m}{(1+x^k)^l}\,dx=\frac1k\cdot B(\frac{m+1}k,l-\frac{m+1}k),\ B(x,y):=\frac{\Gamma(x)\Gamma(y)}{\Gamma(x+y)},\ \Gamma(x)\Gamma(1-x)=\frac{\pi}{\sin{\pi x}}.
\end{equation}
For brevity, we omit the computation of these integrals. Thus, the lemma is proved.
\end{proof}
\begin{remark}\label{Rem4.cn} It is not difficult to see that the positive constant $C_n$ in the expansions \eqref{4.fgh-new} decays as $n\to\infty$: $C_n\sim C/n$. Indeed, the analytic function $z\to\frac1{1+z^{2n}}$ has simple poles at
$$
z_k:=-\sin(\pi k/(2n))+i\cos(\pi k/(2n))
$$
and at least  one of them is at distance $\sim \pi/(2n)$ from the real axis. This explains why the expansions \eqref{4.fgh-new} start to work only for extremely small $\mu$ (hence, extremely
large $\delta$) if $n$ is large enough (see examples below).
\end{remark}
The next lemma is the analogue of Lemma \ref{Lem2.loglog} for this case.
\begin{lemma}\label{Lem4.theta-new} Let $n\in\Bbb N$ and $2n-d>0$. Then the function $\Theta_{d,n}(\delta)$ possesses the following expansion:
\begin{multline}\label{4.1dtheta-new}
\Theta_{d,n}(\delta)=\frac1{(2\pi)^d}\(\frac{\pi\omega(d)}{\sin(\frac{\pi d}{2n})\cdot d^{d/(2n)}(2n-d)^{1-d/(2n)}}\cdot\delta^{\frac d{2n}}-\right.\\-\left.\frac{2n}{2n-d}-\frac{2d^{1+d/(2n)}n^2\sin(\frac{\pi d}{2n})}{\pi\omega(d)(2n-d)^{2+\frac d{2n}}}\cdot\delta^{-\frac d{2n}}\)+
O(\delta^{-d/n})
\end{multline}
as $\delta\to\infty$.
\end{lemma}
The proof of this statement is based on expansions \eqref{4.fgh-new} and consists of
straightforward calculations which are left to the reader.
\par
We are now ready to state the improved version of \eqref{4.1d-new} with a remainder term,
which can be considered as the main result of this section.
\begin{theorem}\label{Th4.multi}Let $n\in\Bbb N$ and $2n-d>0$. Then the following inequality holds for all $(2\pi)^d$-periodic functions with zero mean:
\begin{equation}\label{4.1dBL-new}
\|u\|^2_{C(\Bbb T^d)}\le c_d(n)\|u\|_{L^2}^{2-d/n}\|(-\Dx)^{n/2}u\|_{L^2}^{d/n}-K_d(n)\|u\|^2_{L^2}
\end{equation}
where $c_d(n):=\frac1{(2\pi)^d}\cdot\frac{\pi\omega(d)}{\sin(\frac{\pi d}{2n})\cdot d^{d/(2n)}(2n-d)^{1-d/(2n)}}$ and the constant $K_d(n)\le\frac{2n}{(2\pi)^d(2n-d)}$  can be found from
\begin{equation}\label{4.1dK-new}
K_d(n):=\sup_{\delta\ge1}\{c_d(n)\delta^{d/(2n)}-\Theta_{d,n}(\delta)\}.
\end{equation}
\end{theorem}
\begin{proof} The finiteness of the supremum \eqref{4.1dK-new} is guaranteed by
expansions \eqref{4.1dtheta-new} coupled with the continuity of the function $\Theta_{d,n}$. The
validity of \eqref{4.1dBL-new} then follows immediately from the definitions of
$\Theta_{d,n}$ and $K_d(n)$. The inequality $K_d(n)\le \frac{2n}{(2\pi)^d(2n-d)}$ follows
from the fact, according to \eqref{4.1dtheta-new}, that the limit of the right-hand side
of \eqref{4.1dK-new} as $\delta\to\infty$ is exactly $\frac{2n}{(2\pi)^d(2n-d)}$.
\end{proof}
\begin{remark}\label{Rem4.important} We emphasise that the first constant $c_d(n)$ in \eqref{4.1dBL-new} {\it coincides} with the
 analogous constant \eqref{0.const} for the case of the whole of $\R^d$ for all admissible $d,n\in\Bbb N$. Thus, in the improved form \eqref{4.1dBL-new} of the interpolation inequality, the difference between the two alternative cases discussed in the introduction is now transformed to the question whether or not the second constant $K_d(n)$ is {\it nonnegative}.
\par
If $K_d(n)>0$ (as we will see below, this is true for the 1D case $d=1$ as well as for the multi-dimensional case if $n$ is not large),
the second term in \eqref{4.1dBL-new} is {\it negative} and can be treated as a remainder Brezis-Lieb type term in the usual
interpolation inequality  \eqref{4.1d-new}. In particular, this term can simply be omitted,
which shows that in such a case, the best
constant in \eqref{4.1d-new} in the space periodic case coincides with the analogous constant for $\R^d$. In addition, if $K_d(n)$ is {\it strictly less}
than $\frac{2n}{(2\pi)^d(2n-d)}$, we may conclude that there are exact extremals for \eqref{4.1dBL-new} (this follows from the fact that the third term in the expansions \eqref{4.1dtheta-new} is strictly negative).
\par
By contrast, if $K_d(n)<0$, the lower order term in \eqref{4.1dtheta-new} becomes {\it positive} and cannot be removed without increasing the first constant $c_d(n)$. Thus, according to Theorem \ref{Th4.multi}, adding the positive {\it lower order} corrector to the classical inequality \eqref{4.1d-new} allows us not to increase the constant in the leading term (which remains the same as in the case of $\R^d$). This improvement may in
practice be essential since, in many applications to PDEs, inequalities \eqref{4.1d-new} are used in order to estimate the {\it higher} order norm in situations where the lower order norm is already estimated (say, via the energy inequality, see \cite{Tem}). In that case, only the constant in the leading term is truly essential and the approach with the corrector term allows us not only to
decrease it, but also gives its exact analytical value.
\end{remark}
\begin{remark}\label{Rem4.nontrivial} Note that the possibility of keeping the leading constant the same as in $\R^d$ in the case of bounded domains, just by adding the lower order corrector,
is not an obvious fact and it is specific for the domains without boundary or for Dirichlet boundary conditions. Indeed, let us consider the case where $\Omega=[-\pi,\pi]^d$ with, say, Neumann boundary conditions. Then, because
of the symmetries, the functions $u_\mu(x-\{\pi\}^d)$ (our extremals, but shifted to the `corner'
of the hypercube $\Omega$) will satisfy the boundary conditions. However, only `one quarter'
of the functions are now in the domain $\Omega$, so the $C$-norms of the functions remain unchanged,
but all $H^s$-norms are halved. Thus, the leading constant $c_d(n)$ in \eqref{4.1dBL-new} must
be at least four times larger than for $\R^d$; see also \cite{MShap} for the case of domains with cusps where not only the coefficient, but also the form of the leading term in the asymptotic
expansions, will be different.
\end{remark}
We conclude this subsection by considering the low-dimensional cases including the numerical analysis of the constant $K_d(n)$ for small $n$ and $d$.
\par\noindent
\subsection{The one-dimensional case {\boldmath $d=1$}.} A comprehensive analysis of \eqref{4.1d-new} was given in \cite{Il1}. In particular, as proved there, the best constant in the 1D case of \eqref{4.1d-new} is exactly $c_1(n)$ for all $n\ge1$ and the exact extremals do not exist. We use this information in order to verify that the constant $K_1(n)$ is strictly positive.
\begin{lemma}\label{Lem4.minus} Let $d=1$ and $n\ge1$. Then, the best constant $K_1(n)$ in the remainder term is strictly positive.
\end{lemma}
\begin{proof} The negativity of $K_1(n)$ would contradict the fact that the best constant in \eqref{4.1d-new} is $c_1(n)$, as proved in \cite{Il1}, so we only need to exclude the case $K_1(n)=0$.
\par
Assume now that $K_1(n)=0$. Then, again, in the light of the expansions \eqref{4.1dtheta-new},
the zero supremum in \eqref{4.1dK-new} must be a maximum achieved at some point $\delta=\delta_*<\infty$. Consequently, the associated conditional extremal $u_{\mu(\delta_*)}(x)$ is an {\it exact} extremum function for \eqref{4.1d-new} which contradicts the result of \cite{Il1}. Thus, $K_1(n)>0$ and the lemma is proved.
\end{proof}
\begin{remark}\label{Rem4.exact} Recall that, in the 1D case, each of the functions $f(\mu)$, $g(\mu)$ and $h(\mu)$ can be written in closed form
through the logarithmic derivatives of the Euler $\Gamma$-function using the famous identity
$$
\frac d{dx}\log \Gamma(x)+\gamma=\sum_{k=0}^\infty\(\frac1{k+1}-\frac1{k+x}\)
$$
and by expanding the function $\frac1{1+\mu k^{2n}}$ as a sum of elementary fractions. Although this formula is not very helpful for asymptotic analysis, it may be used for high presicion
numerics since effective ways to compute the logarithmic derivatives for the gamma functions
are known and incorporated in the algebraic manipulation software, for
example, Maple.
\end{remark}

\begin{figure}[ht]
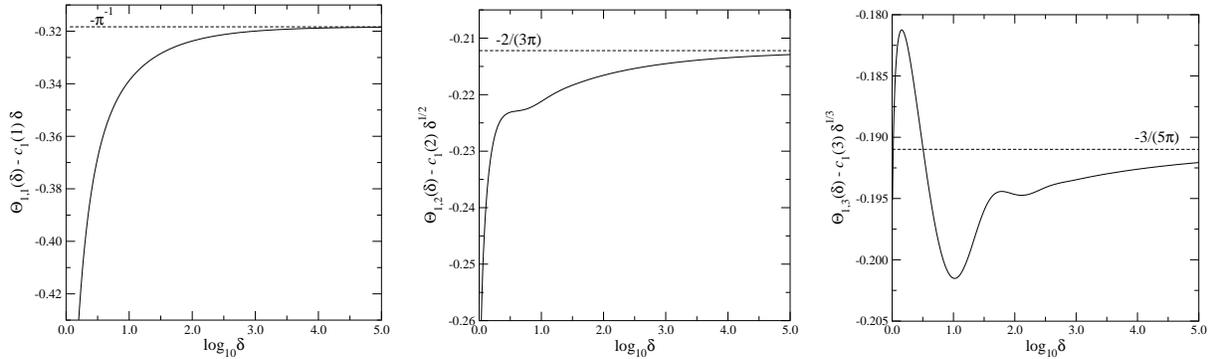

\centering
\begin{tabular}{ccc}
\includegraphics*[angle=0,width=2.0in]{1d_n1.eps} &
\includegraphics*[angle=0,width=2.0in]{1d_n2.eps} &
\includegraphics*[angle=0,width=2.0in]{1d_n3.eps}
\end{tabular}
\caption{Left: 1D case, $n = 1$, middle: $n = 2$, right: $n = 3$.}
\label{pic_1d_n_1_2_3}
\end{figure}

We now present some numerical results for $n$ not large.
\par
Let $n=1$ or $n=2$. Then, as can be seen from figure~\ref{pic_1d_n_1_2_3},
the function
\begin{equation}\label{4.F-new}
F(\delta):=\Theta_n(\delta)-c_1(n)\delta^{1/(2n)}+\frac{n}{\pi(2n-1)}
\end{equation}
is monotone increasing and is {\it negative} for all $\delta\ge0$ (for large $\delta$ this property is obvious since the third term in \eqref{4.1dtheta-new} is negative and tends to zero;
and for $\delta$ not large the numerics is reliable). Thus, we see that $K_1(1)=\frac1\pi$, $K_1(2)=\frac2{3\pi}$ and therefore, the following inequalities hold:
$$
\|u\|^2_{C(\Bbb T^1)}\le \|u\|_{L^2}\|u'\|_{L^2}-\frac1\pi\|u\|^2_{L^2},\ \
  \|u\|^2_{C(\Bbb T^1)}\le \frac{\sqrt2}{\sqrt[4]{27}}\|u\|_{L^2}^{3/2}\|u''\|_{L^2}^{1/2}-\frac2{3\pi}\|u\|^2_{L^2}
$$
and exact extremals for this inequality do not exist. Note also that, in contrast to the
case $n=1$, the graph of $F(\delta)$ becomes {\it non-concave} and we may expect that a
local maximum for small $\delta$ will appear for larger $n$. Indeed, for $n=3$,
in figure~\ref{pic_1d_n_1_2_3}, we see
two local maxima for the function, one of which becomes {\it larger} than zero.
\par
Therefore, $K_1(3)<\frac3{5\pi}\sim 0.19099$ and, in contrast to \eqref{4.1d-new} exact extremals for
the improved version \eqref{4.1dBL-new} exist. In addition, according to our
computations, $K_1(3)\sim 0.181232$ and is achieved at $\delta=1.43404$.
\par
We have observed the analogous phenomenon
for all  larger $n$, so the {\it conjecture} that $K_1(n)< \frac{n}{\pi(2n-1)}$ for
all $n\ge3$, and probably tends to zero as $n\to\infty$, looks reasonable.
\begin{remark}\label{Rem4.1d-new} Thus, even in the simplest one-dimensional case, our method
allows not only the reproduction of known results, but also gives some interesting new information about remainders of the Brezis-Lieb type.
\end{remark}
\subsection{The two-dimensional case {\boldmath $d=2$}.} By contrast to the 1D case,
the situation here
is essentially less understood and, to the best of our knowledge,  the exact value
of $c_2(n)$ was not known even for $n=2$ (note that the inequality $c_2(2)<\frac1\pi$
was established in \cite{IlT} although, as we will see $c_2(2)=\frac14$).
We mention also that the analogous problem on the 2D sphere $\Bbb S^2$ has been studied by Ilyin \cite{Il2} and it was found that for $n\ge8$ the corresponding constant becomes {\it strictly} larger than the analogous constant for $\mathbb R^2$ and can be found only numerically. As we will see, the same phenomenon also occurs on the torus for $n>9$.

\begin{figure}[ht]
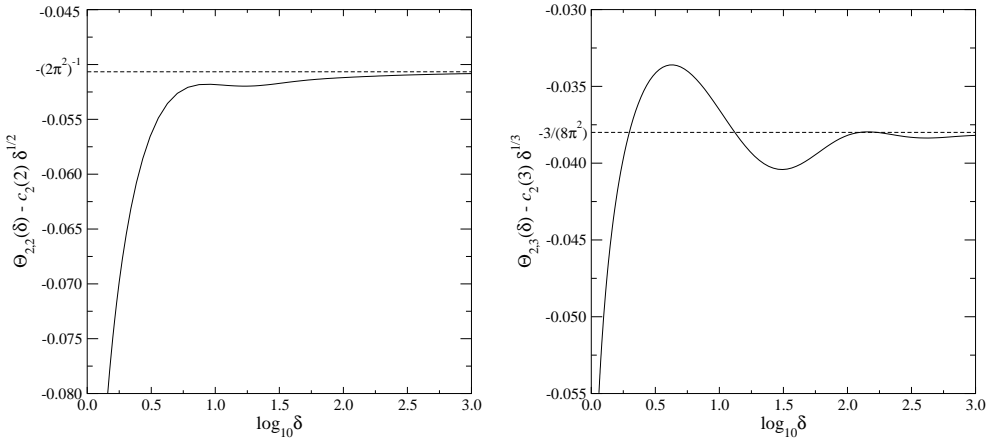

\centering
\begin{tabular}{cc}
\includegraphics*[angle=0,width=2.5in]{2d_n2.eps} &
\includegraphics*[angle=0,width=2.5in]{2d_n3.eps}
\end{tabular}
\caption{Plots of $\Theta_{2, n}(\delta) - c_2(n)\delta^{1/n}$ against $\delta$ in
the 2D case, for $n = 2$ (left), and $n=3$ (right).}
\label{pic_2d_n_2_3}
\end{figure}

\begin{figure}[ht]
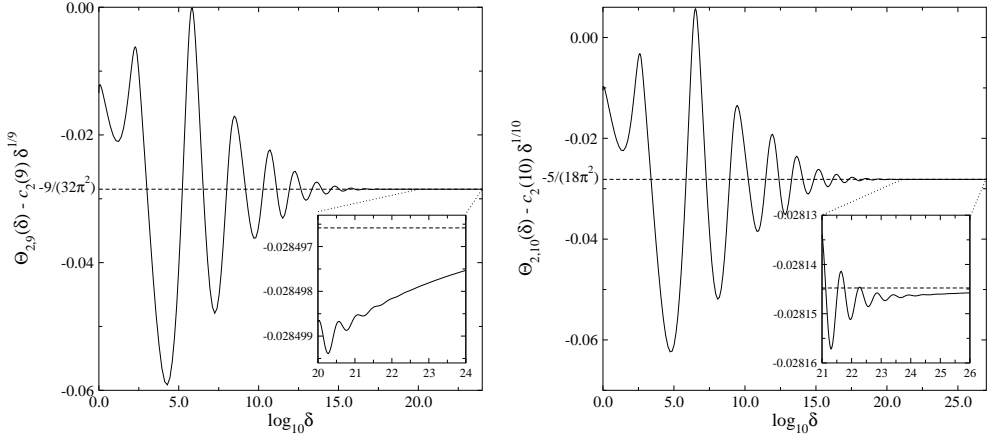

\centering
\begin{tabular}{cc}
\includegraphics*[angle=0,width=2.5in]{2d_n9.eps} &
\includegraphics*[angle=0,width=2.5in]{2d_n10.eps}
\end{tabular}
\caption{Plots of $\Theta_{2, n}(\delta) - c_2(n)\delta^{1/n}$ against $\delta$ in
the 2D case, for $n = 9$ (left), and $n=10$ (right).}
\label{pic_2d_n_9_10}
\end{figure}

We now present our numerical study of the constant $K_2(n)$. Let $n=2$ (the least possible value
in the 2D case). As figures~\ref{pic_2d_n_2_3} and~\ref{pic_2d_n_9_10} show,
the function $F(\delta):=\Theta_n(\delta)-c_2(n)\delta^{1/n}+\frac{n}{4\pi^2(n-1)}$ remains negative (although not monotone increasing; again the negativity of the third term in \eqref{4.1dtheta-new} guarantees negativity for large $\delta$ and we only need to check it for $\delta$ not
large, where the numerics are reliable). Thus, $K_2(2)=\frac{n}{4\pi^2(n-1)}$ and the following
inequality holds:
\begin{equation}\label{4.2dimproved-new}
\|u\|_{C(\Bbb T^2)}^2\le\frac14\|u\|_{L^2}\|\Dx u\|_{L^2}-\frac{1}{2\pi^2}\|u\|^2_{L^2}
\end{equation}
for all $2\pi\times2\pi$-periodic functions with zero mean (and there are no exact extremals for the
inequality).
\par
Now let $n=3$. Then, as we see from figure~\ref{pic_2d_n_2_3}
the function $F$ is {\it positive} for $1.98 \leq\delta\leq 13.2$,
therefore $K_2(3)<\frac3{8\pi^2}$, but still remains positive. Thus, analogously
to the 1D case, exact extremals appear for \eqref{4.2dimproved-new} at $n=3$
and we may compute the sharp value of $K_2(3)$ only numerically.
\par
As our computations show, the positive maximum of $F$ will only grow when $n$ grows, so we expect that this phenomenon holds for all $n\ge3$. In addition, we see that the coefficient $K_2(n)$ remains positive until $n\le9$, but for $n=10$ the value $K_2(10)$ becomes {\it strictly negative}. This means that inequality \eqref{4.1dBL-new} holds no more for $K_2=0$ and we need a {\it positive} lower order corrector in order to be able to use the sharp constant $c_2(n)$ in the leading term.
\begin{remark}\label{Rem4.2d-new} Thus, using our approach, one can not only verify the new inequality \eqref{4.2dimproved-new} where all the constants are the best possible, but
also to prove that the constant $c_2(n)$ can be chosen in an optimal way (coinciding with the analogous constant for $\R^2$) for all $n\ge2$, if a (possibly positive) lower order corrector
is added. The
lower order corrector indeed becomes positive for large $n$ ($n\ge10$) but remains negative
otherwise.
\end{remark}
\par\noindent
\subsection{The three-dimensional case {\boldmath $d=3$}.} In the case $n = 2$,
the numerics show that
$K_3(2)<\frac1{2\pi^3}$, but remains {\it positive}; see figure~\ref{pic_3d_n_2_3_6}
(left), in which we have plotted
$F(\delta) := \Theta(\delta)-\frac{\sqrt2\sqrt[4]3}{6\pi}\delta^{3/4}$.
We find that $K_3\sim\frac{0.996}{2\pi^3}\sim 0.01605$, achieved at $\delta=25.6$, which
gives the exact extremals for problem \eqref{4.1dBL-new}.

\begin{figure}[ht]
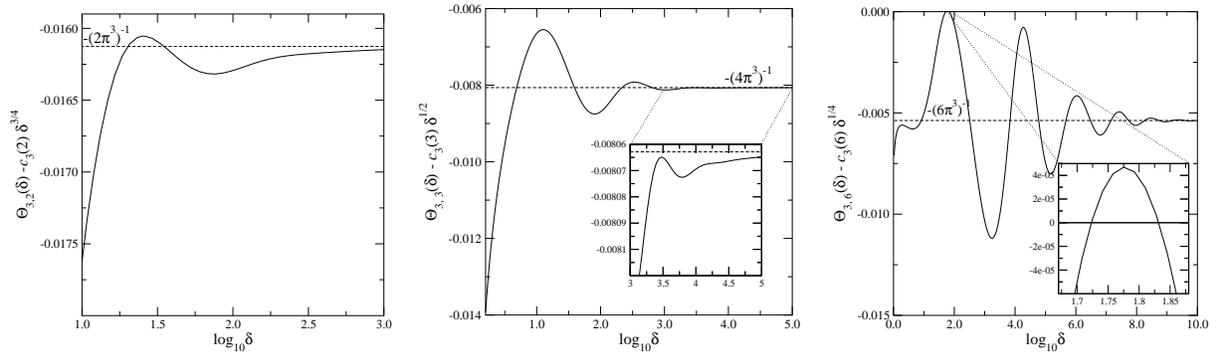

\centering
\begin{tabular}{ccc}
\includegraphics*[angle=0,width=2.0in]{3d_n2.eps} &
\includegraphics*[angle=0,width=2.0in]{3d_n3.eps} &
\includegraphics*[angle=0,width=2.0in]{3d_n6.eps}
\end{tabular}
\caption{Left: 3D, $n = 2$, middle: $n = 3$, right: $n = 6$.}
\label{pic_3d_n_2_3_6}
\end{figure}

\par\noindent
Analogously to the 2D case, the function $F$ becomes more oscillatory when $n$ grows and,
for $n\geq 6$ it crosses the $x$-axis and the second constant $K_3(6)<0$ becomes strictly negative:
see figure~\ref{pic_3d_n_2_3_6} (right).  Actually, $K_3(6)$ is very close to
zero ($K_3(6)\sim -10^{-5}$) but it is already strictly negative.

\section{The large {\boldmath $n$} limit}\label{s5}
As we have seen in the previous section, the asymptotic expansions \eqref{4.1dtheta-new}
and \eqref{4.fgh-new}, which do not contain any oscillatory terms, start to work only for
extremely large $\delta$ (extremely small $\mu$) if $n$ is large enough. In contrast to
this, as the numerics show, the difference between $\Theta_{k,n}(\delta)$ and the leading
term of its expansions
\begin{equation}\label{6.F}
F_{d,n}(\delta):=\Theta_{d,n}(\delta)-c_d(n)\delta^{1/(2n)}
\end{equation}
is highly oscillatory when $\delta$ is not extremely large (and the values of the second
constant $K_d(n)$ in \eqref{4.1dBL-new} are determined exactly by this transient part
if $n$ {\it is} large). The aim of this section is to clarify the nature of this
oscillation by studying the large $n$ limit ($n\to\infty$) of the properly scaled
function \eqref{6.F}. As we know, there is an essential difference between the 1D
and multi-dimensional cases (since, in particular, $K_1(n)$ is always positive in
1D and may be negative in the multi-dimensional case), so we will consider these two
cases separately.

\subsection{The one-dimensional case: regular oscillations} Let us introduce a scaled parameter $z$ such that $\mu:=z^{-2n}$ and write function $F_{d,n}$ as follows:
\begin{equation}\label{6.F1d}
F_{1,n}(z):=\frac1{2\pi}\frac{f(z^{-2n})^2}{g(z^{-2n})}-c_1(n)\(\frac{h(z^{-2n})}{g(z^{-2n})}\)^{1/(2n)}
\end{equation}
and pass to the point-wise limit $n\to\infty$ in every term of this formula. Clearly,
\begin{equation}\label{6.c_1}
c_1(n)=\frac1{\pi}\(1+\frac{\log(2n-1)}{2n}+O(1/n)\) \ \text{and} \ \
\lim_{n\to\infty}c_1(n)=\frac1\pi.
\end{equation}
The following lemma gives the point-wise limit of the other terms in \eqref{6.F1d}.
\begin{lemma}\label{Lem6.per} The point-wise limit
\begin{equation}\label{6.1delta}
\delta_\infty(z):=\lim_{n\to\infty}\(\frac{h(z^{-2n})}{g(z^{-2n})}\)^{1/(2n)},\ \ z\ge1
\end{equation}
as $n\to\infty$ is a continuous piece-wise smooth function given by the following expression:
\begin{equation}\label{6.parab}
\delta_\infty(z)=\begin{cases} l, \ z\in[l,\sqrt{l(l+1)}], \ l\in\mathbb N \\
                 z^2/(l+1),\ z\in[\sqrt{l(l+1)},l+1], \ l\in\mathbb N
                 \end{cases}
\end{equation}
and the limit $n\to\infty$ of the first term on the right-hand side of \eqref{6.F1d} is a
piecewise constant function given by
\begin{equation}\label{6.1theta}
\theta_\infty(z):=\lim_{n\to\infty}\frac{f(z^{-2n})^2}{g(z^{-2n})}=2[z]
\end{equation}
for all non-integer $z$ (here $[z]$ stands for the integer part of $z$).
\end{lemma}
\begin{proof} Let us first check formula \eqref{6.parab}. Clearly,
$\lim_{n\to\infty}(g(z^{-2n}))^{1/(2n)}=1$, so we only need to find the limit of
\begin{equation}\label{6.sum}
\(\sum_{k\in\Bbb Z}\frac{k^{2n}}{(1+(k/z)^{2n})^2}\)^{1/(2n)}.
\end{equation}
Let $z\in(l,l+1)$. Then, for $k\le l$, the $k$th term is approximately $k^{2n}$ and
the largest term corresponds to $k=l$.
For $k\ge l+1$, the denominator becomes large. Neglecting the term $1$,
we see that the $k$th term is close to $(z^2/k)^{2n}$ and the largest term corresponds to
$k=l+1$. Thus,
$$
\lim_{n\to\infty}\(\sum_{k\in\Bbb Z}\frac{k^{2n}}{(1+(k/z)^{2n})^2}\)^{1/(2n)}=\max\{l,z^2/(l+1)\},\ z\in(l,l+1)
$$
which gives \eqref{6.1delta}.
\par
In order to  verify \eqref{6.1theta} for $z\in(l,l+1)$, it is enough to note that in both sums (for $f$ and for $g$) the $k$th term tends to one and to zero if $k\le l$ and $k\ge l+1$ respectively (actually, the limit value is slightly different for integer points, but this is not
important for our purposes). Thus, the lemma is proved.
\end{proof}
\begin{corollary}\label{Cor6.lim1d} Let $z\in(l,l+1)$, $l\in\Bbb N$. Then
\begin{equation}\label{6.F1dlim}
F_{1,\infty}(z):=\lim_{n\to\infty}F_{1,n}(z)=\frac1\pi\begin{cases} 0,\ z\in(l,\sqrt{l(l+1)})\\
                                                                    l-z^2/(l+1),\ z\in[\sqrt{l(l+1)},l+1)
\end{cases}
\end{equation}
and, therefore,
$$
\max\{F_{1,\infty}(z)\}=0,\ \ \inf\{F_{1,\infty}(z)\}=-\frac1\pi
$$
and the infimums are achieved as $z\to l-$, $l=2,3,...$.
\end{corollary}
\begin{corollary}\label{Cor6.lim} The second constant $K_1(n)$ in inequality \eqref{4.1dBL-new} satisfies
\begin{equation}\label{6.lim1d}
\lim_{n\to\infty}K_1(n)=0
\end{equation}
\end{corollary}
From the previous section we know that $K_1(n)\ge0$ and
the limit \eqref{6.F1dlim} then shows that the limit of $K_1(n)$ must be equal to zero.

\begin{figure}[htbp]
\centering
\includegraphics*[angle=0,width=5.0in]{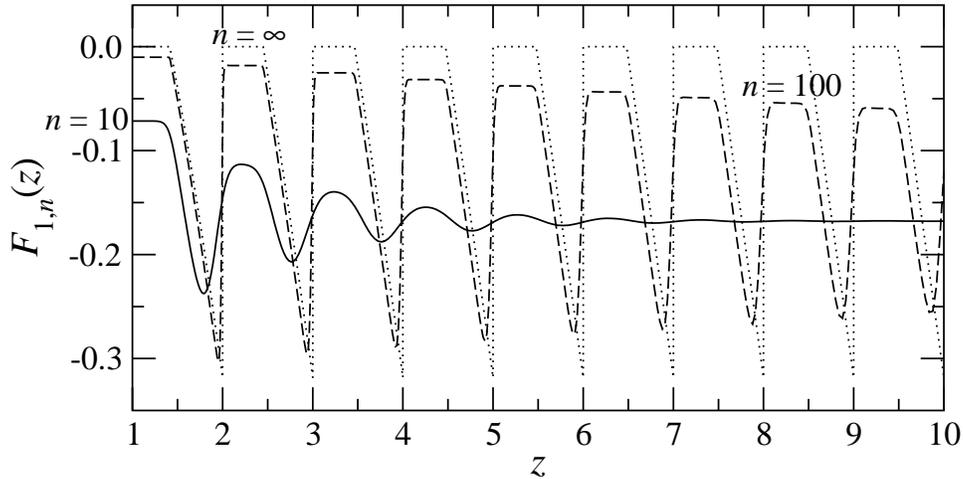}
\caption{Plots of $F_{1, n}(z)$ for $n = 10, 100$ and as
$n\rightarrow\infty$.}
\label{F1n}
\end{figure}

The results of our numerical simulations for $n=10, 100$ and the infinite
limit are shown in figure~\ref{F1n}.
We see that even for the case $n=10$ the limit function $F_{1,\infty}(z)$ allows prediction
of the positions of first maxima and
minima of $F_{1,5}(z)$. For $n=100$, we already see similar oscillations on the whole
interval $z\in[1,10]$ (which covers  the interval $\delta\le 10^{50}$ in the unscaled
variables) and for larger $n$ we also see
quantitative agreement with the limit case. Thus, the limit function $F_{1,\infty}(z)$
encapsulates the nature of {\it regular} oscillations of $F_{1,n}(z)$ for large $n$.

\subsection{The multi-dimensional case: irregular oscillations} We now turn to the multi-dimensional case. In order to avoid technicalities, we concentrate only on the 2D case although the situation is similar for $d>2$. In a fact, the point-wise limit of the function $F_{2,n}(\delta)$ as $n\to\infty$ can be found analogously to the 1D case, but the behaviour of the limit function will be much
more irregular than in the 1D case, for number theoretical reasons. Let us
introduce a slightly different scaling of the parameter $\mu$, namely, $\mu=z^{-n}$
and consider the function
\begin{equation}\label{6.F2d}
F_{2,n}(z):=
\frac1{4\pi^2}\frac{f(z^{-n})^2}{g(z^{-n})}-c_2(n)\(\frac{h(z^{-n})}{g(z^{-n})}\)^{1/n}.
\end{equation}
As in the one dimensional case, let us find the point-wise limit of every term on the right-hand side of \eqref{6.F2d}. First of all, clearly
$$
\lim_{n\to\infty}c_2(n)=\frac1{4\pi}
$$
and the limits of the other terms are given by the following lemma.
\begin{lemma}\label{Lem6.2D} Let $l_1$ and $l_2$ be two successive natural numbers that can
be represented as the sum of two squares of integers, and let $z\in(l_1,l_2)$. Then
\begin{equation}\label{6.2delta}
\delta_\infty(z):=\lim_{n\to\infty}\(\frac{h(z^{-n})}{g(z^{-n})}\)^{1/n}=
\begin{cases} l_1, \ z\in(l_1,\sqrt{l_1l_2}],\\
                 z^2/l_2,\ z\in[\sqrt{l_1l_2},l_2).
                 \end{cases}
\end{equation}
Analogously, the limit $n\to\infty$ of the first term on the right-hand side
of \eqref{6.F2d} is a piecewise constant function given by
\begin{equation}\label{6.2theta}
\theta_\infty(z):=\lim_{n\to\infty}\frac{f(z^{-n})^2}{g(z^{-n})}=R_2(l_1),
\end{equation}
where $R_2(z)$ is the number of integer points $k\in\Bbb Z^2$ such that $|k^2|\le z$,
excluding zero.
\end{lemma}
The proof of this lemma repeats almost word for word the proof of Lemma \ref{Lem6.per}
(replacing `subsequent integers' by `successive integers which can be represented as a sum
of two squares') and for this reason is omitted.
\begin{corollary}\label{Cor6.2dF} Let $l_1$ and $l_2$ be two subsequent integers which can be represented as a sum of two squares and let $z\in(l_1,l_2)$. Then
\begin{equation}\label{6.2dFlim}
F_{2,\infty}(z):=\lim_{n\to\infty}F_{2,n}(z)=\frac1{4\pi^2}\begin{cases}R_2(l_1)-\pi l_1, \ z\in(l_1,\sqrt{l_1l_2}]\\
R_2(l_1)-\pi z^2/l_2,\ z\in[\sqrt{l_1l_2},l_2).
\end{cases}
\end{equation}
\end{corollary}
Thus, in contrast to the 1D case, the limit function $F_{2,\infty}(z)$ contains the function
$R_2(l_1)$  (the number of integer points in a disk of radius $\sqrt{z}$). In addition,
the leading term in the expansion of that function is exactly $\pi z$. It is known that the
remainder $R_2(l_1)-\pi l_1$ is unbounded both from below and from above, is
approximately of order $l_1^{1/4}$, and demonstrates very irregular oscillatory
behaviour for large $l_1$ (see \cite{Hardy}). This explains, in particular, why $K_2(n)$ becomes negative for sufficiently large $n$ as well as suggesting that
$$
\lim_{n\to\infty}K_2(n)=-\infty.
$$
Below we present the results of our numerical simulations for $n=10,25,100$ and $n=\infty$
in figure~\ref{F2n}.

\begin{figure}[htbp]
\centering
\includegraphics*[angle=0,width=5.0in]{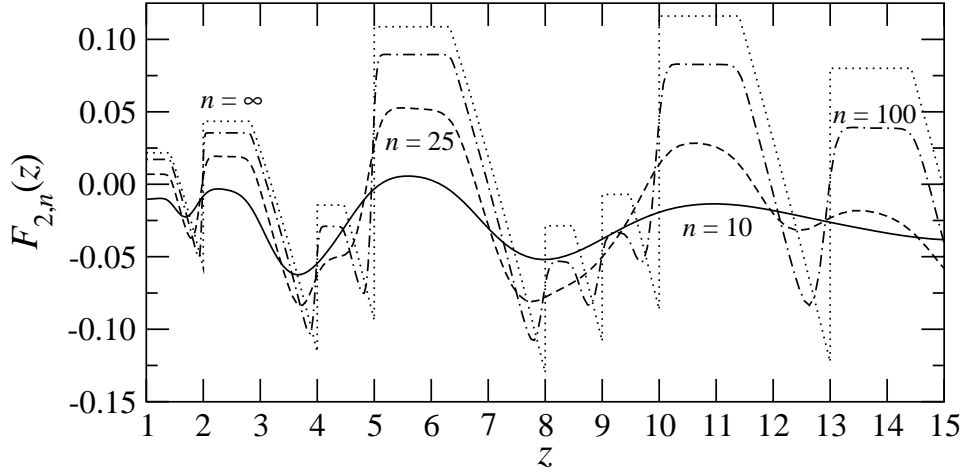}
\caption{Plots of $F_{2, n}(z)$ for $n = 10, 25, 100$ and as
$n\rightarrow\infty$.}
\label{F2n}
\end{figure}

We see from figure~\ref{F2n} that even for $n=10$ (when the graph first crosses the
$x$-axis and $K_2(n)$ becomes strictly negative), the limit function
$F_{2,\infty}(z)$ predicts the positions of maxima and minima
of $F_{2,10}(z)$ and the correspondence with $F_{2,\infty}(z)$ grows as $n$ increases.
Thus, we see that the irregular oscillations of $F_{2,n}(z)$ for large $n$ can be
explained by taking the limit $n\to\infty$, and by the irregularity of the second term in
the asymptotic expansion for the number of integer points in a ball of radius $z$.

\section{Appendix: exact formula for the integration constant $\beta$}
The aim of this Appendix is to find analytically the value of the integration constant
$\beta$ occurring in the asymptotics \eqref{2.fexp} and \eqref{2.gexp}. We start by
recalling the standard technique of estimating sums by integrals, adapted to the 2D case.
\begin{lemma}\label{LemA.1} Let the function $R:\R_+\to\R_+$ be monotone decreasing. Then
\begin{multline}\label{13}
\int_{\Omega}R(|x|)\,dx-4\int_1^\infty R(x)\,dx-4R(1)\le \Sum R(|k|)\le\\\le \int_{\Omega}R(|x|)\,dx+4\int_1^\infty R(x)dx+4R(1)+4R(\sqrt2),
\end{multline}
where $\Omega:=\{(x,y)\in\R^2, \max\{|x|,|y|\}\ge1\}$.
\end{lemma}
\begin{proof} We use the obvious estimate
$$
\int_{C_{(k_1-1,k_2-1)}}R(|x|)\,dx\ge R(|k|)\ge\int_{C_{k_1,k_2}}R(|x|)\,dx,
$$
where the right estimate holds for all $k_i\ge0$ (and $C_k:=[k_1,k_1+1]\times[k_2,k_2+1]$),
 and for the validity of the left estimate, we need $k_i\ge1$. Thus,
\begin{equation}\label{A.1}
\Sum_{k_i\ge0} R(|k|)\ge \int_{\Omega_+}R(|k|)\,dx
\end{equation}
(with $\Omega_+:=\Omega\cap\{x\ge0,\,y\ge0\}$) and
$$
\Sum R(|k|)\ge \int_{\Omega} R(|k|)\,dx- 4\sum_{k=1}^\infty R(k)\ge \int_{\Omega} R(|x|)\,dx-4\int_1^\infty R(k)\,dk-4R(1),
$$
where we have used that
$$
\sum_{k=2}^\infty R(k)\le \int_1^\infty R(x)\,dx.
$$
On the other hand,
$$
\sum_{k_i\ge1,k\ne(1,1)}R(|k|)\le\int_{\Omega_+}R(|x|)\,dx
$$
which together with \eqref{A.1} gives the left-hand side of \thetag{13} and finishes
the proof of the lemma.
\end{proof}
The next lemma gives the formula for $\beta$ in terms of a 2D extension of the Euler constant.
\begin{lemma}\label{LemA.2} The integration constant $\beta$ is the following 2D
analogue of the Euler-Mascheroni constant:
\begin{equation}\label{31}
\beta=\lim_{N\to\infty}\(\Sum_{|k|\le N}\frac1{k^2}-2\pi\log N\).
\end{equation}
\end{lemma}
\begin{proof} We write out the function $f$ in the following form:
$$
f(\mu)=\lim_{N\to\infty}\(\Sum_{|k|\le N}\frac1{k^2}-\Sum_{|k|\le N}\frac\mu{1+\mu k^2}\):=
\lim_{N\to\infty}\(\Sum_{|k|\le N}\frac1{k^2}-\varphi_N(\mu)\)
$$
and find the asymptotic behaviour for $\varphi_N(\mu)$ by replacing the sum with the corresponding
integral using the analogue of estimate \eqref{13}. Indeed, the 1D integrals are
of order $\mu^{1/2}$ uniformly with respect to $N$ and the sum of all terms for which $N-C\le|k|\le N+C$ is also of the order $\mu^{1/2}$ uniformly with respect to $N$; there are at most $cN$
such terms and the sum does not exceed $c\frac {\mu N}{1+\mu N^2}\sim\mu^{1/2}$. Thus,
\begin{equation}\label{32}
f_N(\mu)=\int_{B_N(0)}\frac\mu{1+\mu|x|^2}\,dx+ O(\mu^{1/2})=\pi\log(1+\mu N^2)+O(\mu^{1/2})
\end{equation}
and the remainder is uniformly small with respect to $N$ as $\mu\to0$. This gives
\begin{equation}\label{A.2}
\lim_{N\to\infty}\(2\pi\log N-f_N(\mu)\)=-\pi\lim_{N\to\infty}\log(\mu+\frac1{N^2})+O(\mu^{1/2})=\pi\log\frac1\mu+O(\mu^{1/2}).
\end{equation}
Since, by definition, the integration constant $\beta$ satisfies
$$
\beta=\lim_{\mu\to0}\(f(\mu)-\pi\log\frac1\mu\),
$$
equality \eqref{A.2} gives
$$
f(\mu)-\pi\log\frac1\mu=\lim_{N\to\infty}\(2\pi\log N-f_N(\mu)\)+O(\mu^{1/2})=\beta+O(\mu^{1/2})
$$
and passing to the limit $\mu\to0$, we deduce \eqref{31}. Lemma \ref{LemA.2} is proved.
\end{proof}
\begin{remark}\label{RemA.??} Actually, many of constants of the type \eqref{31} are
explicitly known (e.g., the so-called Madelung constants, etc., see \cite{Fi} and
references therein). However, we failed to find the formula for the constant \eqref{31}
in the literature, so we will prove the analytic expression for it in terms of the usual Euler
constant and the Gamma function in the next lemma, based on the Hardy formula for lattice sums.
\end{remark}
\begin{lemma}\label{LemA.3} The constant $\beta$ can be expressed in terms of the classical Euler-Mascheroni constant $\gamma$ as follows:
\begin{equation}\label{34}
\beta=\pi\gamma+4\beta'(1)
\end{equation}
with $\beta'(1)=\frac\pi4\(\gamma+2\log2+3\log\pi-4\log\Gamma(1/4)\)\sim 0.19290$ (here $\beta(z)$ and $\Gamma(z)$ are the Dirichlet beta and gamma functions respectively.
\end{lemma}
\begin{proof} We use the explicit formula \eqref{3.hardy} for the lattice sums:
\begin{multline}\label{37}
\Sum \frac1{k^{2(1+\eb)}}=4\zeta(1+\eb)\beta(1+\eb)=\\=4(\frac1\eb+\gamma+O(\eb))(\frac\pi4+\beta'(1)\eb+O(\eb^2))=
\frac\pi\eb+\pi\gamma+4\beta'(1)+O(\eb),
\end{multline}
where $\zeta(x)$ is the Riemann zeta function. We also introduce the following notations:
$$
\psi_N:=\Sum_{|k|\le N}\frac1{k^2},\ \ \psi_N(\eb):=\Sum_{|k|\le N}\frac1{k^{2(1+\eb)}},\ \psi(\eb):=\Sum\frac1{k^{2(1+\eb)}}
$$
and compute the expansions for $\psi(\eb)-\psi_N(\eb)=\sum_{|k|\ge N}\frac1{k^{2(1+\eb)}}$ for small $\eb$ and large $N$. As before, it is not difficult to see that replacing the
sum by the integral works and gives
\begin{equation}\label{38}
\psi(\eb)-\psi_N(\eb)=\int_{|x|>N}\frac{dx}{|x|^{2(1+\eb)}}+O(N^{-1})=\frac\pi\eb N^{-2\eb}+O(N^{-1})
\end{equation}
uniformly with respect to $\eb\to0$. Thus,
$$
\lim_{\eb\to0}(\psi(\eb)-\psi_N(\eb)-\frac\pi\eb)=\frac\pi\eb[N^{-2\eb}-1]+O(N^{-1})=-2\pi\log N+O(N^{-1}).
$$
Using also the fact that, for every finite $N$, $\lim_{\eb\to0}\psi_N(\eb)=\psi_N$, we obtain
$$
\psi_N-2\pi\log N=\lim_{\eb\to0}(\psi(\eb)-\frac\pi\eb)+O(N^{-1})
$$
and, thanks to \eqref{37}
$$
\beta=\lim_{\eb\to0}(\psi(\eb)-\frac\pi\eb)=\pi\gamma+4\beta'(1).
$$
Thus, using the known expression for the derivative of the $\beta$-function at $s=1$, we
derive the desired formula \eqref{34} and this finishes the proof of the lemma.
\end{proof}

\end{document}